\documentclass[12pt, reqno]{amsart}
\usepackage[margin=1in]{geometry}
\usepackage{amssymb, amsmath, amsthm, amsfonts}
\usepackage{mathrsfs}
\usepackage[all]{xy}
\usepackage{multirow}

\newtheorem{theorem}{Theorem}[section]
\newtheorem{lemma}[theorem]{Lemma}
\newtheorem{proposition}[theorem]{Proposition}
\newtheorem{corollary}[theorem]{Corollary}
\usepackage{tikz}
\usetikzlibrary{positioning}
\usepackage{verbatim}

\newtheorem{predefinition}[theorem]{Definition}
\newenvironment{definition}{\begin{predefinition}\rm}{\end{predefinition}}
\newtheorem{preremark}[theorem]{Remark}
\newenvironment{remark}{\begin{preremark}\rm}{\end{preremark}}
\newtheorem{prenotation}[theorem]{Notation}

\newtheorem{preexample}[theorem]{Example}
\newenvironment{example}{\begin{preexample}\rm}{\end{preexample}}
\newtheorem{preclaim}[theorem]{Claim}

\newtheorem{prequestion}[theorem]{Question}
\newenvironment{question}{\begin{prequestion}\rm}{\end{prequestion}}
 \makeatletter

\DeclareMathOperator{\Gal}{Gal}

\DeclareMathOperator{\rk}{rk}

\DeclareMathOperator{\Jac}{Jac}

\DeclareMathOperator{\Proj}{Proj}

\newcommand{\bbF}{\mathbb{F}}

\usepackage{color}
\definecolor{jolimauve}{rgb}{0.67,0.1,0.7}

\title{Non-ordinary curves with a Prym variety of low $p$-rank}
\date{}

\author[Celik]{Turku Ozlum Celik}
\address{Turku Ozlum Celik, Laboratoire IRMAR, UMR CNRS 6625 Campus de Beaulieu, 35042, Rennes, France}
\email{turku-ozlum.celik@univ-rennes1.fr}

\author[Elias]{Yara Elias}
\address{Yara Elias, Max Planck Institute for Mathematics, Vivatsgasse 7, 53111
Bonn, Germany}
\email{yara.elias@mail.mcgill.ca}

\author[G\"une\c{s}]{Bur\c{c}\.{i}n G\"une\c{s}}
\address{Bur\c{c}\.{i}n G\"une\c{s}, Sabanci University, Faculty of Engineering and Natural Sciences, 
Tuzla, Istanbul, 34956, Turkey}
\email{bgunes@sabanciuniv.edu}

\author[Newton]{Rachel Newton}
\address{Rachel Newton\\ Department of Mathematics and Statistics\\ University of Reading\\
Whiteknights\\ PO Box 220\\
Reading RG6 6AX\\
UK}
\email{r.d.newton@reading.ac.uk}

\author[Ozman]{Ekin Ozman}
\address{Ekin Ozman, Bogazici University, Faculty of Arts and Sciences, Bebek, Istanbul, 34342, Turkey}
\email{ekin.ozman@boun.edu.tr}

\author[Pries]{Rachel Pries}
\address{Rachel Pries, Department of Mathematics, 
Colorado State University, 
Fort Collins, CO 80523, USA}
\email{pries@math.colostate.edu}

\author[Thomas]{Lara Thomas}
\address{Lara Thomas, Universit\'e de Franche Comt\'e, 16, route de Gray, 25030 Besan\c{c}on, France and Lyc\'ee Claude Fauriel, 28, avenue de la Lib\'eration, 42000 Saint-\' Etienne, France.}
\email{lthomas@math.cnrs.fr}

\begin{document}

\begin{abstract} 
If $\pi: Y \to X$ is an unramified double cover of a smooth curve of genus $g$,
then the Prym variety $P_\pi$ is a principally polarized abelian variety of dimension $g-1$.
When $X$ is defined over an algebraically closed field $k$ of characteristic $p$, 
it is not known in general which $p$-ranks can occur for $P_\pi$ 
under restrictions on the $p$-rank of $X$.
In this paper, when $X$ is a non-hyperelliptic curve of genus $g=3$, 
we analyze the relationship between the Hasse-Witt matrices of $X$ and $P_\pi$.
As an application, when $p \equiv 5 \bmod 6$, we prove that
there exists a curve $X$ of genus $3$ and $p$-rank $f=3$ 
having an unramified double cover $\pi:Y \to X$ for which $P_\pi$ 
has $p$-rank $0$ (and is thus supersingular);
for $3 \leq p \leq 19$, we verify the same for each $0 \leq f \leq 3$.
Using theoretical results about $p$-rank stratifications of moduli spaces, 
we prove, for small $p$ and arbitrary $g \geq 3$, that there 
exists an unramified double cover $\pi: Y \to X$ such that both $X$ and $P_\pi$ have small $p$-rank.

\

\noindent
Keywords: curve, Jacobian, Prym variety, abelian variety, p-rank, supersingular, 
moduli space, Kummer surface\\
MSC: primary 11G20, 14H10, 14H40, 14K15, 14Q05\\
secondary: 11G10, 11M38, 14G17, 14K25, 14Q10
\end{abstract}



\maketitle

\section{Introduction} \label{Introduction}
Let $p$ be a prime number and let $k$ be an algebraically closed field of characteristic $p$. 
Let $A$ be an abelian variety of dimension $g$ defined over $k$.  
The \emph{$p$-rank} of $A$ is the integer $f$ defined by $\#A[p](k) = p^f$. 
It is known that $0 \leq f \leq g$. 
Let $X$ be a smooth projective connected curve of genus $g$ defined over $k$. 
Then the $p$-rank of $X$ is the $p$-rank of its Jacobian.
An equivalent definition is that $f$ equals the maximal integer $m$ such that there exists an unramified
$(\mathbb Z/p \mathbb Z)^m$-Galois cover $X' \to X$.
When $f=g$, we say that $A$ (or $X$) is \emph{ordinary}.
 
The $p$-rank of a curve $X$ equals the stable rank of the Frobenius map on $\mathrm{H}^1(X,\mathcal{O}_X)$, and thus can be determined from its Hasse-Witt or Cartier-Manin matrix (see subsections~\ref{SCM}-\ref{prankHW}). 
Given a prime $p$ and integers $g$ and $f$ with $0 \leq f \leq g$, a result of Faber and 
Van der Geer
\cite[Theorem 2.3]{FaberVdG2004} implies that
there exists a curve over $\overline{{\mathbb F}}_p$ of genus $g$ and $p$-rank $f$.

We assume that $p$ is odd from now on. 
Consider an unramified double cover
$$\pi: Y \longrightarrow X.$$
Then $\mathrm{Jac}(Y)$ is isogenous to $\mathrm{Jac}(X) \oplus P_{\pi}$
where $P_\pi$ is the \emph{Prym variety} of $\pi$.  
In this context, $P_\pi$ is a principally polarized abelian variety of dimension $g-1$.
The $p$-rank $f'$ of $P_{\pi}$ satisfies $0 \leq f' \leq g-1$.
Since the $p$-rank is an isogeny invariant, the $p$-rank of $Y$ equals $f+f'$. 

Now the following question arises naturally.
\begin{question} \label{Qnatural}
Suppose that $p$ is an odd prime, and $g,f,f'$ are integers such that $g \geq 2$, $0 \leq f \leq g$, and $0 \leq f' \leq g-1$. 
Does there exist a curve $X$ defined over ${\overline{\mathbb F}}_p$ of genus $g$ and $p$-rank $f$  
having an unramified double cover $\pi: Y \longrightarrow X$ such that 
$P_{\pi}$ has $p$-rank $f'$?
\end{question}

The answer to Question \ref{Qnatural} is yes for $p \geq 3$ and $0 \leq f \leq g$ 
under the following restrictions:\\
\begin{itemize}
\item[-] when $g=2$ \cite[Proposition 6.1]{OzmanPries2016},
unless $p=3$ and $f=0,1$ and $f'=0$, in which case the answer is no \cite[Example 7.1]{FaberVdG2004}; 
\item[-] when $g \geq 3$ and $f'=g-1$, as a special case of \cite[Theorem 1.1(1)]{OzmanPries2016};
\item[-] when $g \geq 3$ and $f'=g-2$ (with $f \geq 2$ when $p=3$), 
by \cite[Theorem 7.1]{OzmanPries2016};
\item[-] when $p\geq 5$ and $g \geq 4$ and $\dfrac{g}{2} -1 \leq f' \leq g-3$, 
by \cite[Corollary 7.3]{OzmanPries2016}. 
\end{itemize}

In this paper, we study the first open case of Question \ref{Qnatural}, 
which occurs when $X$ has genus $g=3$ and $P_\pi$ has $p$-rank $0$. 
We focus on the case that $X$ is a smooth plane quartic or, equivalently, that $X$ is not
hyperelliptic.  
(If $X$ has genus $g=3$, the Riemann-Hurwitz formula implies 
that the genus of $Y$ is $5$ and hence the dimension of $P_\pi$ is $2$.) 

Given an unramified double cover $\pi:Y \to X$ of a smooth plane quartic $X$, 
in Lemma~\ref{forwardHW}
we use work of Bruin~\cite{Bruin2008} to analyze the Hasse-Witt matrices of $X$ and $P_\pi$ simultaneously, in terms of the quadratic forms that determine $\pi$.
As an application, we verify that the answer to Question~\ref{Qnatural} is yes when $g=3$ and $3 \leq p \leq 19$ in Proposition~\ref{R3f0_forwardmethod}. 

Given a genus $2$ curve $Z:z^2=D(x)$,
it is possible to describe all smooth plane quartic curves $X$ having an unramified double cover
$\pi:Y \to X$ whose Prym variety $P_\pi$ is isomorphic to ${\rm Jac}(Z)$. 
Specifically, 
the Kummer variety $K={\rm Jac}(Z)/\langle -1 \rangle$ of ${\rm Jac}(Z)$
is a quartic surface in ${\mathbb P}^3$.
Each smooth plane quartic $X$ having an unramified double cover $\pi:Y \to X$ such that 
$P_\pi \simeq {\rm Jac}(Z)$
arises as the intersection $V \cap K$ for some plane $V \subset {\mathbb P}^3$
\cite{Mumford, Verra, CasselsFlynn}. 
Building on work of Kudo and Harashita \cite{Kudo}, we provide a method to determine the Hasse-Witt matrix of $X$ from $V$ and $Z$ in Proposition \ref{prop:Hasse-Witt}.

Suppose that $\pi: Y \to X$ is an unramified double cover with $P_\pi \simeq {\rm Jac}(Z)$ as in the previous paragraph.
In Section~\ref{sec:backwards}, we first choose $Z:z^2=x^6-1$
and verify in Proposition~\ref{Pexist5mod6} that the answer to Question~\ref{Qnatural} is yes when $(g,f,f')=(3,3,0)$ and $p \equiv 5 \bmod 6$. 

In the second part of Section~\ref{sec:backwards}, for an arbitrary smooth curve $Z$ of genus $2$, we use commutative algebra to analyze the condition that 
(*) $X$ is non-ordinary and $Z$ has $p$-rank $0$.
In Proposition \ref{Phomogeneous}, we prove that condition (*) 
is equivalent to the vanishing of $4$ homogeneous polynomials of degree $(p+1)(p-1)/2$ 
in the coefficients on $D(x)$ and the vanishing of one homogeneous polynomial of degree $6(p-1)$ in the coefficients of $D(x)$ and $V$. 
As an application, when $p=3$, we give an explicit characterization of the curves $Z$ and planes $V$ for which $\pi:Y \to X$ satisfies condition (*), see Section \ref{S20}. 

Finally, we apply the new results described above for genus $3$ curves 
in small characteristic to study the $p$-ranks of Prym varieties of smooth curves of 
arbitrary genus $g \geq 3$. For this, we use an inductive method developed in \cite{AP:08}.
This yields Corollary \ref{Cmaincor}, which extends \cite[Corollary 7.3]{OzmanPries2016}
for small $p$ and gives the following application.

\begin{corollary} \label{Cintro}
Let $3 \leq p \leq 19$. 
The answer to Question \ref{Qnatural} is yes, for any $g \geq 2$, under the following conditions on $(f,f')$: 
\begin{enumerate}
\item If $g=3r$ and $(f, f')$ is such that $2 r \leq f \leq g$ and $r -1 \leq f' \leq g-1$;
\item If $g=3r + 2$ and $(f,f')$ is such that 
$2r \leq f \leq g$ and $r \leq f' \leq g-1$, 
(with $f \geq 2r+2$ when $p=3$);
\item If $g=3r + 4$ and $(f, f')$ is such that 
$2 r \leq f \leq g$ 
(with $f \geq 2r+4$ when $p=3$)
and $r+1 \leq f' \leq g-1$.
\end{enumerate}
\end{corollary}

All of the existence results for $p$-ranks described above are proven using a 
geometric analysis of the $p$-rank stratification of 
moduli spaces of curves and their unramified covers.
For example, \cite[Theorem 2.3]{FaberVdG2004} 
shows that the $p$-rank $f$ stratum $\mathcal{M}_g^f$ of $\mathcal{M}_g$ is non-empty and each component 
has dimension $2g-3+f$, see also \cite[Section 3]{AP:08}.

Consider the moduli space $\mathcal{R}_g$ whose points represent unramified double covers 
$\pi: Y \longrightarrow X$, where $X$ is a smooth curve of genus $g$.
Let $\mathcal{R}_g^{(f,f')}$ denote the stratum of $\mathcal{R}_g$ representing points where $X$ has $p$-rank $f$ and $P_\pi$
has $p$-rank $f'$.
To answer Question \ref{Qnatural}, it suffices to prove that $\mathcal{R}_g^{(f,f')}$ is non-empty in characteristic $p$.

Suppose that $\mathcal{R}_g^{(f,f')}$ is non-empty; then one can study its dimension.  
As an application of purity results for the Newton polygon stratification, \cite[Proposition 5.2]{OzmanPries2016} shows that:
if $\mathcal{R}_g^{(f,f')}$ is non-empty, then each of its components
has dimension at least $g-2+f+f'$.

In fact, the dimension of $\mathcal{R}_g^{(f,f')}$ attains this lower bound in the following cases:
\begin{itemize}
\item[-] when $f'=g-1$, then each component of $\mathcal{R}_g^{(f,f')}$ has dimension $2g-3+f$
as a special case of \cite[Theorem 1.1(1)]{OzmanPries2016};
\item[-] when $f'=g-2$, with $f \geq 2$ when $p=3$, 
then each component of $\mathcal{R}_g^{(f,f')}$ has dimension $2g-4+f$
\cite[Theorem 7.1]{OzmanPries2016};
\item[-] when $p \geq 5$ and $\dfrac{g}{2} -1 \leq f' \leq g-3$, 
then at least one component of $\mathcal{R}_g^{(f,f')}$ has dimension $g-2+f +f'$
\cite[Corollary 7.3]{OzmanPries2016}.
\end{itemize}
We extend these geometric results in Section \ref{sec:induction}.
In Theorem \ref{Tbigg}, for any prime $p$, we prove an inductive result that allows one to leverage information when $g=3$ about 
$\mathcal{R}_3^{(f,0)}$ into information about $\mathcal{R}_g^{(f,f')}$ for arbitrarily large $g$.
The final result is Corollary \ref{Cmaincor}; it allows us to prove the existence of unramified double covers $\pi:Y \to X$ with control over the $p$-rank $f$ of $X$ and the $p$-rank $f'$ of $P_\pi$ as long as $f$ is bigger than approximately $2g/3$ and $f'$ is bigger than approximately $g/3$.

\begin{corollary} 
If $3 \leq p \leq 19$, the stratum $\mathcal{R}_g^{(f, f')}$ 
has a (non-empty) component of dimension $g-2 +f +f'$
for all $g \geq 2$ under the conditions on $(f,f')$ found in Corollary \ref{Cintro}.
\end{corollary}

The condition on $p$ is needed to show that $\mathcal{R}_3^{(2,0)}$ has dimension $3$; more specifically, that
there is a 3-dimensional family of smooth plane quartics $X$ with $p$-rank $2$ 
having an unramified double cover $\pi:Y \to X$ such that $P_\pi$ has $p$-rank $0$.
We expect this to be true for all odd primes $p$ but were only able to prove it computationally for $3 \leq p \leq 19$.

\subsection{Outline of the paper}

Here is the material contained in each section.

Section 2: Definitions and background material.

Section 3: The Hasse-Witt matrices of $X$ and $P_\pi$ in terms of quadratic forms.  Examples of unramified double covers $\pi:Y \to X$ with given $p$-ranks 
$(f,f')$ for small $p$.

Section 4: The Hasse-Witt matrix of a plane quartic $X$ defined as the intersection of a 
plane and quartic surface in ${\mathbb P}^3$.

Section 5: The Kummer surface of an abelian surface ${\rm Jac}(Z)$, information about the determinant 
of the Hasse-Witt matrix from Section 4, 
the non-emptiness of ${\mathcal R}_3^{(3,0)}$ when $p \equiv 5 \bmod 6$, 
and information about the geometry of ${\mathcal R}_3^{(2,0)}$ when $p=3$. 

Section 6: An auxiliary result about the number of points on the Kummer surface of a supersingular curve of genus $2$ defined over a finite field.

Section 7: Results about $p$-ranks of Pryms of unramified double covers of curves 
of arbitrary genus for small $p$ proven inductively from results in Sections 3.

\subsection*{Acknowledgements}
This project began at the \emph{Women in Numbers Europe 2} workshop in the Lorentz Center, Leiden. We are very grateful to the Lorentz Center for their hospitality and support. Elias benefited from a Leibniz fellowship at the Oberwolfach Research Institute during part of this work. Ozman was partially supported by by Bogazici University Research Fund Grant Number 10842 and  by the BAGEP Award of the Science Academy with funding supplied by Mehve\c{s} Demiren in memory of Selim Demiren. Pries was partially supported by NSF grant DMS-15-02227.  We would like to thank Bouw and Bruin for helpful conversations.

\section{Background} \label{sec:background}

\subsection{Notation}

Let $k$ be an algebraically closed field of characteristic $p >0$.
Unless stated otherwise, every curve is a smooth projective connected $k$-curve.
Suppose that $C$ is a curve of genus $g \geq 1$.

\subsection{The Cartier-Manin matrix} \label{SCM}

Let $L$ be the function field of $C/k$. Since $k$ is perfect, there exists a separating variable $x\in L\setminus k$ such that $L/k(x)$ is algebraic and separable. It follows that $L=L^p(x)$ and hence every element $z\in L$ can be written uniquely in the form
\[z=z_0^p+z_1^px+\dots +z_{p-1}^px^{p-1}\]
with $z_0,\dots, z_{p-1}\in L$. The \emph{Cartier operator} $\mathscr{C}$ is defined on differentials of the first kind by
\[\mathscr{C}((z_0^p+z_1^px+\dots +z_{p-1}^px^{p-1})dx)=z_{p-1}dx.\]
The Cartier operator is $\frac{1}{p}$-linear, meaning that $\mathscr{C}(a^p\omega_1+b^p\omega_2)=a\mathscr{C}(\omega_1)+b\mathscr{C}(\omega_2)$ for all $a,b\in L$ and all $\omega_1,\omega_2\in \Omega^1(L)$. 
It is independent of the choice of separating variable
and hence gives a well-defined map on the $k$-vector space of regular differentials on $C$, $\mathscr{C}:\mathrm{H}^0(C,\Omega^1_C)\to \mathrm{H}^0(C,\Omega^1_C)$.

\begin{definition} \label{Dcartiermanin}
Let $\omega_1,\dots, \omega_g$ be a $k$-basis for $\mathrm{H}^0(C,\Omega^1_C)$. Write $\mathscr{C}(\omega_j)=\sum_{i=1}^g c_{ij}\omega_i$ with $c_{ij}\in k$.
The Cartier-Manin matrix of $C$ with respect to the basis $\omega_1,\dots, \omega_g$ is the matrix $(c_{ij}^p)_{i,j}$.
\end{definition}

\begin{remark}
The Cartier-Manin matrix depends on the choice of basis. Let $\omega_1',\dots,\omega_g'$ be another $k$-basis for 
$\mathrm{H}^0(C,\Omega^1_C)$ and let $T=(t_{ij})$ be the change of basis matrix so that $\omega_j=\sum_{i=1}^g{t_{ij} \omega_i'}$. Then the Cartier-Manin matrix with respect to the basis $\omega_1',\dots,\omega_g'$ is $T^{(p)}(c_{ij}^p)T^{-1}$, where $T^{(p)}$ denotes the matrix obtained from $T$ by taking the $p$th power of each of its entries.
\end{remark}

\subsubsection{The Cartier-Manin matrix of a hyperelliptic curve} \label{2.2}

Let $p$ be odd and let $Z$ be a hyperelliptic curve of genus $g$. Then $Z$ has an equation of the form $z^2=f(x)$ for a separable polynomial $f(x) \in k[x]$ having degree $2g+1$ or $2g+2$. 
Write $\omega_i=\frac{x^{i-1}}{y}dx$, so that $\{\omega_1, \ldots, \omega_g\}$
is a basis for $\mathrm{H}^0(Z,\Omega^1_Z)$.
  
\begin{proposition} (Yui) \cite[Proposition 2.3]{Yui}
\label{prop:yui}
Let $c_s$ denote the coefficient of $x^s$ in the expansion of $f(x)^{(p-1)/2}$. Then the Cartier-Manin matrix of $Z$ is $A_0=(c_{ip-j})_{i,j}$. 
\end{proposition}

\subsection{The Hasse-Witt matrix} \label{SHW}

The \emph{(absolute) Frobenius} $F$ of $C$ is the morphism of schemes given by the identity on the underlying topological space and $f\mapsto f^p$ on $\mathcal{O}_C$. We write $F^*$ for the induced endomorphism of $\mathrm{H}^1(C,\mathcal{O}_C)$. It is a $p$-linear map, meaning that $F^*(\lambda \xi)=\lambda^pF^*\xi$ for all $\lambda\in k$ and all $\xi\in \mathrm{H}^1(C,\mathcal{O}_C)$. 

\begin{proposition} \label{prop:duality} \cite[Proposition 9]{Serre}
Serre duality gives a perfect pairing \[\langle\ ,\ \rangle:\mathrm{H}^1(C,\mathcal{O}_C)\times\mathrm{H}^0(C,\Omega^1_C)\to k\] such that
\[ \langle F^* \xi,\omega \rangle = \langle \xi, \mathscr{C} \omega \rangle^p \]
for all $\xi\in \mathrm{H}^1(C,\mathcal{O}_C)$ and all $\omega\in \mathrm{H}^0(C,\Omega^1_C)$. 
\end{proposition}

\begin{definition}
Let $\xi_1, \dots \xi_g$ be a $k$-basis of $ \mathrm{H}^1(C,\mathcal{O}_C)$.
 Write $F^*(\xi_j)=\sum_{i=1}^g a_{ij}\xi_i$ with $a_{ij}\in k$. The Hasse-Witt matrix of $C$ with respect to the basis $\xi_1, \dots \xi_g$ is the matrix $(a_{ij})_{i,j}$.
 \end{definition}
 
\begin{remark}
 The Hasse-Witt matrix depends on the choice of basis. Let $\xi_1',\dots,\xi_g'$ be another $k$-basis for 
$\mathrm{H}^1(C,\mathcal{O}_C)$ and let $S=(s_{ij})$ be the change of basis matrix so that $\xi_j'=\sum_{i=1}^g{s_{ij} \xi_i}$. Then the Hasse-Witt matrix with respect to the basis $\xi_1',\dots,\xi_g'$ is $S^{-1}(a_{ij})S^{(p)}$, where $S^{(p)}$ denotes the matrix obtained from $S$ by taking the $p$th power of each of its entries.
 \end{remark}
 
\begin{remark}\label{rmk:dual} If the basis $\xi_1, \dots \xi_g$ of $\mathrm{H}^1(C,\mathcal{O}_C)$ is the dual basis 
for the basis $\omega_1, \dots, \omega_g$ of $\mathrm{H}^0(C,\Omega^1_C)$, then the Hasse-Witt matrix is the transpose of the Cartier-Manin matrix. 
\end{remark}

\subsection{The $p$-rank}\label{prankHW}

If $A$ is an abelian variety of dimension $g$ over $k$, its $p$-rank is the number $f_A$ such that $\#A[p](k) = p^{f_A}$.
If $C$ is a curve of genus $g$ over $k$, its $p$-rank is the $p$-rank of ${\rm Jac}(C)$.
We write $f_A$ (resp.\ $f_C$) for the $p$-rank of $A$ (resp.\ $C$).

Here is another definition of the $p$-rank.
The $k$-vector space $\mathrm{H}^1(C,\mathcal{O}_C)$ has a direct sum decomposition into $F^*$-stable subspaces as
\[\mathrm{H}^1(C,\mathcal{O}_C)=\mathrm{H}^1(C,\mathcal{O}_C)_{s}\oplus \mathrm{H}^1(C,\mathcal{O}_C)_n\]
where $F^*$ is bijective on $\mathrm{H}^1(C,\mathcal{O}_C)_s$ and nilpotent on $\mathrm{H}^1(C,\mathcal{O}_C)_n$. The dimension of $\mathrm{H}^1(C,\mathcal{O}_C)_s$ is equal to the rank of the 
composition of $F^*$ with itself $g$ times, and this rank is called the \emph{stable} rank of Frobenius on $\mathrm{H}^1(C,\mathcal{O}_C)$.

\begin{proposition}\label{prop:stable}
The $p$-rank of $C$ is equal to the stable rank of the Frobenius endomorphism $F^*:\mathrm{H}^1(C,\mathcal{O}_C)\to \mathrm{H}^1(C,\mathcal{O}_C)$.
\end{proposition}

\begin{proof}
See \cite{Serre}.
\end{proof}

The $p$-rank of the Jacobian of $C$ can be determined from either
the Cartier-Manin or Hasse-Witt matrix. 
For a matrix $M$, we write $M^{(p^i)}$ for the matrix obtained from $M$ by raising each of its entries to the power $p^i$. 

\begin{proposition}\label{prop:prank}
Let $C$ be a curve of genus $g$ with Hasse-Witt matrix $H$ and Cartier-Manin matrix $M$. Then the $p$-rank of $C$ is
\[f_C=\rk (HH^{(p)}\dots H^{(p^{g-1})})=\rk (M^{(p^{g-1})}\dots M^{(p)}M). \]
\end{proposition}

\begin{proof}
The first equality follows from Proposition~\ref{prop:stable} and the fact that Frobenius is $p$-linear. The second equality is a consequence of Serre duality, see Proposition~\ref{prop:duality}. 
\end{proof}

\begin{remark}
When $g=1$ or $g=2$, then $C$ has $p$-rank $0$ if and only if $C$ is supersingular.  When $g \geq 3$, 
there exist curves of $p$-rank $0$ which are not supersingular.
\end{remark}

\subsection{Prym varieties}

Suppose that $p$ is odd.
If $X$ is a curve of genus $g$ defined over $k$, then ${\rm Jac}(X)$ is a principally polarized abelian variety of 
dimension $g$.
There is a bijection between $2$-torsion points on ${\rm Jac}(X)$ and unramified double covers $\pi:Y \to X$.
Without further comment, we require that $Y$ is connected, which is equivalent to the $2$-torsion point being 
non-trivial.  Also, we note that $Y$ is smooth if $X$ is smooth.

Let $\pi : Y \to X$ be an unramified double cover of $X$. 
By the Riemann-Hurwitz formula, $Y$ has genus $2g-1$.
Also $\rm{Jac}(X)$ is isogenous to a sub-abelian variety of ${\rm Jac}(Y)$.
Let $\sigma$ be the endomorphism of ${\rm Jac}(Y)$ induced by the involution generating ${\rm Gal}(\pi)$, 

The {\it Prym variety} $P_\pi$ is the connected component containing $0$ in the kernel of the map 
$\pi^* : {\rm Jac}(Y)\rightarrow {\rm Jac}(X)$.  It is also the image of the map $1 - \sigma$ in ${\rm Jac}(Y)$. 
In other words, 
\[P_{\pi}={\rm Im}(1-\sigma)={\rm Ker}(1+\sigma)^0.\] 
The canonical principal polarization of ${\rm Jac}(Y)$ induces a principal polarization on $P_\pi$. 
Finally, ${\rm Jac}(Y)$ is isogenous to ${\rm Jac}(X) \oplus P_\pi$.

\subsection{Moduli spaces}

In this paper, we consider:
\begin{align*}
\mathcal{M}_g &  \textrm{ the moduli space of curves of genus }g \textrm{ over } k;\\
\mathcal{A}_g &  \textrm{ the moduli space of principally polarized abelian varieties of dimension }g \textrm{ over } k;\\
\mathcal{R}_g & \textrm{ the moduli space whose points represent unramified double covers $\pi:Y\to X$ over $k$},\\ & \textrm{where $X$ is a curve of genus }g;\\
\mathcal{R}_g^{(f,f')} & \textrm{ the stratum of $\mathcal{R}_g$ representing points where $X$ has $p$-rank $f$ and $P_\pi$ has $p$-rank $f'$.}
\end{align*}

\section{Hasse-Witt matrices of genus $3$ curves and their Prym varieties }\label{sec:forwards}

We continue to work over an algebraically closed field $k$ of characteristic $p >2$.
Suppose $\pi: Y \to X$ is an unramified double cover of a non-hyperelliptic smooth curve of genus $3$.
In \cite{Bruin2008}, Bruin describes the equations for $X$ and $P_\pi$ in terms of quadratic forms.
We describe the Hasse-Witt matrices of $X$ and $P_\pi$ in terms of the quadratic forms, 
using results of St\"{o}hr and Voloch in \cite{StohrVoloch} and Yui in \cite{Yui}.
As an application, we answer Question \ref{Qnatural} affirmatively when $3 \leq p \leq 19$ and $g=3$ in 
Proposition \ref{R3f0_forwardmethod}.

\subsection{The Prym variety of an unramified double cover of a plane quartic}
\renewcommand{\arraystretch}{1.4}

A smooth curve $X$ of genus $3$ which is not hyperelliptic is isomorphic to a smooth plane quartic.

\begin{lemma}  \label{Lbruinsetup} \cite[Bruin]{Bruin2008} 
Suppose $\pi:Y \to X$ is an unramified double cover of a smooth plane quartic curve.
Then there exist quadratic forms $Q_1, Q_2, Q_3 \in k[u, v, w]$
such that $X \subset \mathbb{P}^2$ is given by the equation
\begin{equation} \label{EforX}
X : Q_1(u,v,w)Q_3(u,v,w) = Q_2(u,v,w)^2,
\end{equation}
$Y\subset \mathbb{P}^4$ is given by the equations
\begin{equation} \label{EforY}
Y: Q_1(u,v,w) = r^2, \ Q_2 ( u , v , w ) = rs, \ Q_3(u, v, w) = s^2,
\end{equation}
and  the Prym variety $P_\pi$ is isomorphic to ${\rm Jac}(Z)$ for the smooth genus $2$ curve $Z$ with equation
\begin{equation} \label{EforZ}
Z: z^2=D(x):=  -\det(M_1 + 2xM_2 + x^2M_3),
\end{equation}
where $M_i$ is the symmetric $3 \times 3$ matrix such that 
\[
(u,v,w)M_i (u,v,w)^T = Q_i(u,v,w).
\]

Conversely, if $Q_1,Q_2, Q_3 \in k[u,v,w]$ are quadratic forms such that \eqref{EforX} defines
a smooth plane quartic $X$, 
then the equations above give an unramified double cover $\pi:Y \to X$ and a smooth genus $2$ curve $Z$
such that $P_\pi \simeq {\rm Jac}(Z)$.
\end{lemma}

\begin{proof}
This is proven in \cite[Theorem 5.1(4)]{Bruin2008}.
The fact that $Z$ is smooth when $X$ is smooth can be found in \cite[Section 5, Case 4]{Bruin2008}.
\end{proof}

\subsection{Hasse-Witt matrices}

\begin{lemma} \label{forwardHW}
Let $\pi:Y \to X$ be an unramified double cover of a smooth plane quartic curve and suppose 
$P_\pi={\rm Jac}(Z)$.
Let $Q_1,Q_2,Q_3 \in k[u,v,w]$ be quadratic forms as in Lemma \ref{Lbruinsetup}, 
 and
let $D(x) \in k[x]$ be defined as in Lemma \ref{Lbruinsetup}\eqref{EforZ}.
\begin{enumerate}
\item Let \[q(u,v) = Q_2(u,v,1)^2-Q_1(u,v,1)Q_3(u,v,1).\]
Let $a_{i,j}$ be the values in $k$ such that $q(u,v)^{p-1}=\sum_{i,j}{a_{i,j}u^iv^j}$.
Then the Hasse-Witt matrix of $X$ is 
\[H_X=
\begin{pmatrix}a_{p-1,p-1}&a_{2p-1,p-1}& a_{p-1,2p-1}\\
a_{p-2,p-1}&a_{2p-2,p-1}& a_{p-2,2p-1}\\
a_{p-1,p-2}&a_{2p-1,p-2}& a_{p-1,2p-2}
\end{pmatrix}.
\]

\item Let $b_i \in k$ be the values in $k$ such that $D(x)^{(p-1)/2} = \sum_i b_i x^i$.
Then the Hasse-Witt matrix of $Z$ is 
\[H_Z=
\begin{pmatrix}
b_{p-1} & b_{2p-1}\\
b_{p-2} & b_{2p-2}
\end{pmatrix}.
\]
\end{enumerate}
\end{lemma}

\begin{remark}
In Lemma \ref{forwardHW}(1), the 
Hasse-Witt matrix is taken with respect to the basis of $\mathrm{H}^1(X, {\mathcal O_X})$ given by the dual of the basis 
$\frac{du}{q_v}$, $u\frac{du}{q_v}$, $v\frac{du}{q_v}$ of $\mathrm{H}^0(X, \Omega^1_X)$.
In Lemma \ref{forwardHW}(2), the 
Hasse-Witt matrix is taken with respect to the basis of $\mathrm{H}^1(Z, {\mathcal O_Z})$ given by the dual of the basis 
$\frac{dx}{z}, x\frac{dx}{z}$ of $\mathrm{H}^0(Z, \Omega^1_Z)$.
\end{remark}

\begin{proof}
\begin{enumerate}
\item
Let $\omega_1,\dots, \omega_g$ be a basis for $\mathrm{H}^0(X,\Omega_X)$ 
and suppose that the action of the Cartier operator is given by
\begin{equation}\label{Cartierforcurve}
\mathscr{C}(\omega_i)=\sum_{j=1}^g{c_{ij}\omega_j}.
\end{equation}
By Definition \ref{Dcartiermanin} and Remark~\ref{rmk:dual}, 
the Hasse-Witt matrix with respect to the dual basis is the matrix $(c_{ij}^p)$.

The result \cite[Theorem 1.1]{StohrVoloch} of St\"{o}hr and Voloch yields the following information
in \eqref{SVCartier} and \eqref{nabla} about 
the action of the Cartier operator on the smooth plane curve $X$, with affine equation $q(u,v)=0$.
Consider the partial derivative operator \mbox{$\nabla=\frac{\partial^{2p-2}}{\partial u^{p-1}\partial v^{p-1}}$}.
Then for any $h\in k(u,v)$, 
\begin{equation}\label{SVCartier}
\mathscr{C}\left(h\frac{du}{q_v}\right)=\left(\nabla(q^{p-1}h)\right)^{\frac{1}{p}}\frac{du}{q_v}.
\end{equation}
Also, if $\alpha_{i,j} \in k$, then
\begin{equation}\label{nabla}
\nabla\left(\sum_{i,j}{\alpha_{i,j}u^iv^j}\right)=\sum_{i,j}{\alpha_{ip+p-1,jp+p-1}u^{ip}v^{jp}}.
\end{equation}
Write $\omega_i=h_i(u,v)\frac{du}{q_v}$.
By \eqref{SVCartier} and \eqref{Cartierforcurve}, 
\begin{equation}\label{nablaCartier}
\nabla(q^{p-1}h_i)=\sum_{j}{c_{ij}^ph_j^p}.
\end{equation}

In this case, a basis for $H^0(X, \Omega^1_X)$ is 
$\omega_1=\frac{du}{q_v}, \omega_2=u\frac{du}{q_v}, \omega_3=v\frac{du}{q_v}$.
By definition, $q(u,v)^{p-1}=\sum_{i,j}{a_{i,j}u^iv^j}$. By \eqref{nablaCartier} and \eqref{nabla} we have
\[\nabla(q^{p-1})=\sum_{i,j}{a_{ip+p-1,jp+p-1}u^{ip}v^{jp}}=c_{11}^p+c_{12}^pu^p+c_{13}^pv^p\]
where $c_{11},c_{12},c_{13}$ are the entries in the first row of the Hasse-Witt matrix. 
Note that $\deg(f) = 4$, so $\deg(q^{p-1}) = 4(p-1)$ and hence $\deg(\nabla(q^{p-1}))\leq 2(p-1)$. 
Therefore, the coefficient of $u^{ip}v^{jp}$ in $\nabla(q^{p-1})$ is zero unless $i+j\leq 1$. 
Equating the nonzero coefficients gives $c_{11}=a_{p-1,p-1}, c_{12}=a_{2p-1,p-1}$ and $c_{13}=a_{p-1,2p-1}$. 

Similarly, for the other two rows in the Hasse-Witt matrix,
\[
\nabla(q^{p-1}u)=\sum_{i,j}{a_{ip+p-1,jp+p-1}u^{ip+1}v^{jp}}=c_{21}^p+c_{22}^pu^p+c_{23}^pv^p,\]
and
\[\nabla(q^{p-1}v)=\sum_{i,j}{a_{ip+p-1,jp+p-1}u^{ip}v^{jp+1}}=c_{31}^p+c_{32}^pu^p+c_{33}^pv^p.\]

\item Note that $Z$ is smooth since $X$ is smooth by \cite[Section 5, Case 4]{Bruin2008}.
The result follows from \cite[Lemma 5.1]{Bouw}.
Alternatively, the matrix $H_Z$ is the transpose of the Cartier-Manin matrix for $Z$ from 
\cite[Proposition~\ref{prop:yui}]{Yui}.
\end{enumerate}
\end{proof}

\subsection{The $p$-ranks of $X$ and $Z$}\label{tables}

By Proposition~\ref{prop:prank}, 
the $p$-rank $f=f_X$ of $X$ is the rank of $H_XH_X^{(p)}H_X^{(p^2)}$ and the $p$-rank $f'=f_Z$ of $Z$ is the rank of $H_ZH_Z^{(p)}$. 

\begin{proposition}\label{R3f0_forwardmethod}
Let $3 \leq p \leq 19$.
For each pair $(f,f')$ such that $0 \leq f \leq 3$ and $0 \leq f' \leq 2$, there exists an unramified double cover
$\pi:Y \to X$ such that $X$ is a smooth curve of genus $3$ and $p$-rank $f$ and $P_\pi$ has $p$-rank $f'$;
in other words, $\mathcal R_3^{(f,f')}$  is non empty when $3 \leq p \leq 19$.
\end{proposition}

\begin{proof}
The result holds (without any restriction on $p$) when $f'=2$ or $f'=1$ by \cite[Proposition 6.4]{OzmanPries2016}, 
as long as $(f,f') \not = (0,1), (1,1)$ when $p=3$.
To complete the proof, we provide an example below in each case when $f'=0$ 
(and when $p=3$ and $(f,f') = (0,1), (1,1)$).
These examples were found with a computational search, using Lemma \ref{forwardHW}.
\end{proof}

In the examples below, we give the equations of the curves $X,Z$ along with  the coefficients of the quadratic forms that lead to these curves in the following format 
\[[q_{111}, q_{112}, q_{122}, q_{113}, q_{123}, q_{133},q_{211},q_{222}, q_{233}, q_{311}, q_{312}, q_{322}, q_{313}, q_{323}, q_{333}],\]
where: 
\begin{itemize}
	\item $Q_1= q_{111}u^2+ q_{112}uv+ q_{122}v^2+ q_{113}uw+ q_{123}vw+ q_{133}w^2$;
	\item $Q_2= q_{211}u^2+ q_{222}v^2+ q_{233}w^2$;
	\item $Q_3= q_{311}u^2+ q_{312}uv+ q_{322}v^2+ q_{313}uw+ q_{323}vw+ q_{333}w^2$.
\end{itemize}

\begin{example} \label{Equadp3}
$p=3$

\small
\begin{tabular}{|l|l|}
\hline $(f,f')$ & $X, Z,[q_{ijk}]$ \\ \hline \hline
\multirow{2}{*}{$(3,0)$}
& $X: 2u^4 + 2u^3v + u^3 + 2u^2v^2 + u^2v + 2u^2 + 2uv^3 + uv^2 + uv + 2u + v^3 + v^2 + 2v + 1 $  \\ \cline{2-2}
& $ Z: 2x^5 + x^4 + 2x^2 + x + z^2 + 1$  \\ \cline{2-2}
&  $[q_{ijk}]=[ 2, 0, 2, 0, 0, 1, 1, 1, 1, 0, 1, 2, 2, 2, 2] $ \\ \hline
\multirow{2}{*}{$(2,0)$}
& $X: 2u^4 + u^3v + 2u^3 + u^2v + 2uv^3 + uv^2 + 2v^3 + 2v^2 + 2$ \\ \cline{2-2}
&  $Z: x^6 + 2x^5 + 2x^4 + x^2 + x + z^2$  \\  \cline{2-2}
& $[q_{ijk}]=[ 1, 0, 2, 0, 0, 0, 1, 1, 1, 0, 1, 2, 2, 1, 2]$ \\ \hline
\multirow{2}{*}{$(1,0)$} 
& $X: 2u^4 + 2u^3 + 2u^2v + 2u^2 + uv^2 + 2uv + x + 2v^4 + v^3 + v + 2$ \\ \cline{2-2}
& $Z: 2x^6 + 2x^5 + z^2 + 1$\\  \cline{2-2}
& $[q_{ijk}]=[ 2, 0, 0, 1, 1, 1, 1, 1, 1, 0, 0, 1, 1, 1, 0 ] $ \\ \hline 
\multirow{2}{*}{$(0,0)$}
& $X: 2u^4 + 2u^3 + 2u^2v + 2u^2 + 2uv^2 + u + v^4 + 2v^3 + v + 1$  \\ \cline{2-2}
& $Z: 2x^6 + x + z^2 + 1$  \\ \cline{2-2}
& $[q_{ijk}]= [ 2, 0, 2, 0, 0, 1, 1, 1, 1, 0, 0, 1, 1, 1, 2] $ \\ \hline
\multirow{2}{*}{$(1,1)$} 
& $Z: 2x^6 + 2x^4 + x + y^2$ \\ \cline{2-2}
& $X: 2x^4 + x^2y^2 + x^2 + xy^3 + xy^2 + 2xy + 2x + 2y^4 + y^3 + 2y $ \\ \cline{2-2}
& $[q_{ijk}]= [ 0, 0, 1, 0, 0, 2, 1, 1, 1, 0, 1, 0, 1, 1, 2]$ \\ \hline
\multirow{2}{*}{$(0,1)$}
& $Z: 2x^6 + 2x^3 + 2x^2 + x + y^2 + 1 $ \\ \cline{2-2}
& $X: 2x^4 + 2x^3y + 2x^3 + 2x^2 + xy^3 + xy^2 + 2xy + 2x + y^2 + 2$ \\ \cline{2-2}
& $[q_{ijk}]= [ 2, 0, 1, 0, 0, 2, 1, 0, 0, 0, 1, 0, 1, 0, 1]$ \\ \hline
\end{tabular}
\end{example}

\begin{example} $p=5$

\small
\begin{tabular}{|l|l|}
\hline $(f,f')$ & $X, Z,[q_{ijk}]$ \\ \hline \hline
\multirow{2}{*}{$(3,0)$}
& $ X:  4u^4 + 3u^3 + 4u^2v^2 + u^2v + 3uv^2 + 4u + v^3 + 3v^2 + 3v$  \\ \cline{2-2}
&  $ Z: 4x^6 + x^3 + 2x + z^2 + 3 $ \\ \cline{2-2}
& $q_{ijk}=[ 1, 0, 1, 0, 0, 3, 1, 1, 0, 0, 0, 1, 3, 1, 0]$ \\  \hline
\multirow{2}{*}{$(2,0)$} 
& $ X:4u^4 + 3u^3 + 4u^2v^2 + u^2v + 3uv^2 + u + v^3 + 2v^2 + 2v $ \\ \cline{2-2}
&  $Z:4x^6 + 4x^3 + 3x + z^2 + 2$  \\  \cline{2-2}
& $[q_{ijk}]=[ 1, 0, 1, 0, 0, 2, 1, 1, 0, 0, 0, 1, 3, 1, 0] $ \\ \hline
\multirow{2}{*}{$(1,0)$}
& $X:4u^4 + 3u^2v^2 + 3u^2v + 2u^2 + 4uv^2 + uv + 2u + 4v^4 + 4v^3 + 4v^2 + 3$  \\ \cline{2-2}
& $Z: 2x^5 + x^3 + 2x^2 + 2x + z^2 + 2$ \\ \cline{2-2} 
& $[q_{ijk}]= [ 3, 4, 4, 4, 4, 3, 1, 1, 1, 0, 0, 0, 0, 1, 3 ]$ \\
\multirow{2}{*}{$(0,0)$}
& $X: 4u^4 + 3u^2v^2 + 3u^2v + 2u^2 + 4uv^2 + 3uv + 3v + 4v^4 + 4v^3 + v + 2 $\\ \cline{2-2}
&  $Z: 2x^5 + 2x^2 + 2x + z^2 + 2$ \\ \cline{2-2} 
&  $[q_{ijk}]=[ 3, 4, 4, 1, 0, 1, 1, 1, 1, 0, 0, 0, 0, 1, 3]$ \\
\hline
\end{tabular}
\end{example}

\begin{example}
$p=7$

\small
\begin{tabular}{|l|l|}
\hline$(f,f')$ & $X, Z,[q_{ijk}]$ \\ \hline \hline
\multirow{2}{*}{$(3,0)$}
& $X:  6u^4 + 5u^2v^2 + 3u^2v + 6u^2 + 6v^4 + v^3 + 3v^2 + v + 4$  \\ \cline{2-2}  
& $Z:6x^5 + 6x^3 + z^2 + 4  $  \\ \cline{2-2}  
& $[q_{ijk}]=[ 1, 0, 5, 0, 0, 5, 1, 1, 1, 0, 0, 0, 0, 3, 1]$ \\ \hline
\multirow{2}{*}{$(2,0)$}
& $X: 6u^4 + 5u^2v^2 + 2u^2v + 2u^2 + 6v^4 + 4v^3 + 6v^2 + 6$ \\ \cline{2-2} 
&  $Z: 5x^5 + x^4 + 4x^3 + 6x^2 + 4x + z^2 $  \\ \cline{2-2}  
& $[q_{ijk}]= [ 1, 0, 2, 0, 0, 0, 1, 1, 1, 0, 0, 0, 0, 2, 4]$ \\ \hline
\multirow{2}{*}{$(1, 0)$}
& $X:6u^4 + 5u^2v^2 + 2u^2v + u^2 + 6v^4 + 5v^2 + 4v + 5$   \\ \cline{2-2} 
& $Z:6x^5 + 6x^4 + x^2 + x + z^2$  \\ \cline{2-2}  
& $[q_{ijk}]=[ 3, 0, 0, 0, 0, 6, 1, 1, 1, 0, 0, 0, 0, 3, 1 ] $ \\ \hline
\multirow{2}{*}{$(0, 0)$}
& $X: 6u^4 + u^2v^2 + 4u^2 + 3v^4 + 6v^2 + 6$  \\ \cline{2-2}  
& $Z: 4x^5 + 4x^4 + 3x^2 + 3x + z^2 $ \\ \cline{2-2}  
& $[q_{ijk}]=[ 3, 0, 4, 0, 0, 0, 1, 1, 1, 0, 0, 1, 0, 0, 2]$ \\ \hline
\end{tabular}
\end{example}

\begin{example}
$p=11$

\small
\begin{tabular}{|l|l|}
\hline $(f,f')$ & $X,Z,[q_{ijk}]$ \\ \hline \hline
\multirow{2}{*}{$(3,0)$} 
& $X: 10u^4 + 9u^2v^2 + 5u^2v + 2u^2 + 10v^4 + 10v^3 + 4v^2 + 4v + 6$ \\ \cline{2-2}
&  $Z: 9x^5 + 4x^4 + x^3 + 7x^2 + 8x + z^2 + 3$  \\ \cline{2-2}
& $[q_{ijk}]= [ 8, 0, 5, 0, 0, 2, 1, 1, 0, 0, 0, 0, 0, 2, 3] $ \\ \hline
\multirow{2}{*}{$(2,0)$}
& $X: 10u^4 + 9u^2v^2 + 9u^2v + 9u^2 + 10v^4 + v^3 + v^2 + 7v + 7$    \\ \cline{2-2} 
& $Z: 9x^5 + 9x^4 + 9x^3 + 2x^2 + 2x + z^2 + 1$ \\ \cline{2-2} 
& $[q_{ijk}]= [ 10, 0, 6, 0, 0, 9, 1, 1, 0, 0, 0, 0, 0, 2, 2] $ \\ \hline
\multirow{2}{*}{$(1,0)$}
& $X: 10u^4 + 9u^2v^2 + 9u^2v + 8u^2 + 10v^4 + 3v^3 + 10v^2 + 4v + 6	$ \\ \cline{2-2} 
&  $Z: 9x^5 + 2x^4 + 3x^3 + 9x^2 + 2x + z^2 + 8$  \\ \cline{2-2}
& $[q_{ijk}]= [ 10, 0, 7, 0, 0, 2, 1, 1, 0, 0, 0, 0, 0, 2, 3] $ \\ \hline
\multirow{2}{*}{$(0,0)$}
&  $X: 10u^4 + 9u^2v^2 + 3u^2v + 5u^2 + 10v^4 + 9v^3 + 8v^2 + 4v + 1$ \\ \cline{2-2} 
& $Z: 9x^5 + 8x^4 + 9x^3 + 3x^2 + 10x + z^2 + 8$  \\ \cline{2-2} 
& $[q_{ijk}]= [ 7, 0, 10, 0, 0, 2, 1, 1, 1, 0, 0, 0, 0, 2, 1] $ \\ \hline
\end{tabular}
\end{example}

\begin{example}
$p=13$

\small
\begin{tabular}{|l|l|}
\hline $(f,f')$ & $X,Z,[q_{ijk}]$ \\ \hline \hline

\multirow{2}{*}{$(3,0)$} 
& $X: 12u^4 + 11u^2v^2 + 6u^2v + 11u^2 + 12v^4 + 12v^3 + 11v^2 + 3v + 12$ \\ \cline{2-2}
& $Z: 11x^5 + 10x^4 + 8x^3 + 3x^2 + 11x + z^2 + 1 $ \\ \cline{2-2}
& $[q_{ijk}]= [ 3, 0, 6, 0, 0, 8, 1, 1, 1, 0, 0, 0, 0, 2, 0 ] $ \\ \hline

\multirow{2}{*}{$(2,0)$}
& $X:12u^4 + 11u^2v^2 + 2u^2v + 4u^2 + 12v^4 + 11v^3 + 5v^2 + 11v + 6$ \\ \cline{2-2} 
& $Z: 11x^5 + 10x^4 + 8x^3 + 3x^2 + 11x + z^2 + 1 $ \\ \cline{2-2} 
& $[q_{ijk}]= [ 1, 0, 12, 0, 0, 12, 1, 1, 1, 0, 0, 0, 0, 2, 6 ]$ \\ \hline 

\multirow{2}{*}{$(1,0)$}
& $X: 12u^4 + 11u^2v^2 + 9u^2v + 9u^2 + 12v^4 + 7v^3 + 8v^2 + 4v + 1 $ \\ \cline{2-2} 
& $Z: 11x^5 + 7x^4 + 11x^3 + 6x^2 + 5x + z^2 + 1$ \\ \cline{2-2} 
& $[q_{ijk}]= [ 2, 0, 3, 0, 0, 11, 1, 1, 1, 0, 0, 0, 0, 11, 12 ]$ \\ \hline

\multirow{2}{*}{$(0,0)$}
& $X: 12u^4 + 11 u^2 v^2 + 9 u^2v + 7 u^2 + 12 v^4 + 8 v^3 + 6 v^2 + 12 v + 11 $ \\ \cline{2-2} 
& $Z: 6 x^5 + 5 x^4 + 3 x^3 + 6 x^2 + 6 x + z^2 + 6$ \\ \cline{2-2} 
& $[q_{ijk}]=[ 9, 0, 8, 0, 0, 12, 1, 1, 1, 0, 0, 0, 0, 1, 1]$ \\ \hline
\end{tabular}
\end{example}

\begin{example}
$p=17$

\small
\begin{tabular}{|l|l|}
\hline $(f,f')$ & $X,Z,[q_{ijk}]$ \\ \hline \hline

\multirow{2}{*}{$(3,0)$} 
& $X: 16u^4 + 15u^2v^2 + 15u^2 + 16v^4 + 5v^3 + 7v^2 + 6v + 3$  \\ \cline{2-2}
&  $Z: 4x^5 + 8x^4 + 9x^3 + 9x^2 + x + z^2 $ \\ \cline{2-2}
& $[q_{ijk}]= [ 0, 0, 13, 0, 0, 2, 1, 1, 1, 0, 0, 0, 0, 3, 2 ]$ \\ \hline

\multirow{2}{*}{$(2,0)$}
&$X: 16u^4 + 15u^2v^2 + 10u^2v + u^2 + 16v^4 + 3v^3 + 4v^2 + 16 $   \\ \cline{2-2} 
& $Z: 4x^5 + 8x^4 + 9x^3 + 9x^2 + x + z^2 $  \\ \cline{2-2} 
& $[q_{ijk}]= [ 9, 0, 1, 0, 0, 0, 1, 1, 1, 0, 0, 0, 0, 3, 6 ] $ \\ \hline

\multirow{2}{*}{$(1,0)$}
& $X: 	16u^4 + 15u^2v^2 + 10u^2v + 3u^2 + 16v^4 + 4v^3 + 14v + 6$ \\ \cline{2-2} 
&  $Z: 4x^5 + 7x^4 + 5x^3 + 10x^2 + 9x + z^2 + 5$  \\  \cline{2-2} 
& $[q_{ijk}]=[ 9, 0, 7, 0, 0, 16, 1, 1, 1, 0, 0, 0, 0, 3, 10 ]$ \\ \hline

\multirow{2}{*}{$(0,0)$}
& $X:16u^4 + 15u^2v^2 + 6u^2v + 15u^2 + 16v^4 + 9v^3 + 15v^2 + 15v + 16 $ \\ \cline{2-2} 
& $Z: 8x^5 + 7x^4 + 8x^3 + x^2 + 14x + z^2 + 11 $ \\  \cline{2-2}
& $[q_{ijk}]=[ 6, 0, 9, 0, 0, 15, 1, 1, 1, 0, 0, 0, 0, 1, 0]$  \\ \hline 
\end{tabular}

\begin{example}
$p=19$

\small
\begin{tabular}{|l|l|}
\hline $(f,f')$ & $X,Z,[q_{ijk}]$ \\ \hline \hline

\multirow{2}{*}{$(3,0)$} 
& $X: 18u^4 + 17u^2v^2 + 9u^2v + 3u^2 + 18v^4 + 5v^3 + 5v^2 + 18$  \\ \cline{2-2}
&  $Z: 5x^5 + 11x^4 + 13x^3 + 8x^2 + 10x + z^2$ \\ \cline{2-2}
& $[q_{ijk}]= [ 3, 0, 8, 0, 0, 0, 1, 1, 1, 0, 0, 0, 0, 3, 8 ] $ \\ \hline

\multirow{2}{*}{$(2,0)$}
& $X: 18u^4 + 17u^2v^2 + 12u^2v + 18u^2 + 18v^4 + 18v^3 + 9v^2 + 18$ \\ \cline{2-2} 
&  $Z:5x^5 + 11x^4 + 13x^3 + 8x^2 + 10x + z^2$  \\ \cline{2-2} 
& $[q_{ijk}]= [ 4, 0, 6, 0, 0, 0, 1, 1, 1, 0, 0, 0, 0, 3, 5]$ \\ \hline

\multirow{2}{*}{$(1,0)$}
& $X: 	18u^4 + 17u^2v^2 + 5u^2v + 18u^2 + 18v^4 + 6v^3 + 3v^2 + 6v + 4$  \\ \cline{2-2} 
&  $Z: 17x^5 + x^4 + x^3 + 18x^2 + 10x + z^2 + 13 $ \\  \cline{2-2} 
& $[q_{ijk}]= [ 12, 0, 3, 0, 0, 3, 1, 1, 1, 0, 0, 0, 0, 2, 8]$ \\ \hline

\multirow{2}{*}{$(0,0)$}
& $X: 18u^4 + 17u^2v^2 + 17u^2v + 4u^2 + 18v^4 + v^3 + 14v^2 + 12v + 1$ \\ \cline{2-2} 
&  $Z: 16x^5 + 9x^4 + 14x^3 + 10x^2 + 8x + z^2 + 3 $ \\  \cline{2-2} 
& $[q_{ijk}] = [ 11, 0, 4, 0, 0, 10, 1, 1, 1, 0, 0, 0, 0, 5, 4]$ \\ \hline
\end{tabular}
\end{example}
\end{example}

\begin{remark}
Since $k$ is an algebraically closed field of odd characteristic $p$, it is possible to 
diagonalize the quadratic form $Q_2$ and take its coefficients to be $0$ or $1$.
Even so, the complicated nature of the entries of $H_X$ and $H_Z$ makes it difficult to analyze the $p$-ranks algebraically.

The entries of $H_X$ are quite complicated even in terms of the coefficients of $q(u,v)= Q_2(u,v,1)^2-Q_1(u,v,1)Q_3(u,v,1)$.
For example, if $p=3$ and $q(u,v)=\sum_{i,j}{b_{ij}u^iv^j}$, 
then the upper left entry of $H_X$ is $2b_{00}b_{22} + 2b_{01}b_{21} + 2b_{02}b_{20} + 2b_{10}b_{12} + b_{11}^2$. 

Similarly, even the equation for $Z:z^2=D(x)$ is rather complicated in terms of the coefficients of $Q_1$ and $Q_3$.
\end{remark}

\section{The Hasse-Witt matrix of a smooth plane quartic defined as an intersection in ${\mathbb P}^3$} \label{sec:Frobenius}

In this section, we determine the Hasse-Witt matrix of a curve $C$ of genus $3$ defined as the intersection of a plane 
and degree 4 hypersurface in ${\mathbb P}^3$.
We use this result in Section~\ref{sec:backwards} to determine the Hasse-Witt matrix of each 
smooth plane quartic $X$ which has an unramified double cover $\pi:Y \to X$ such that $P_\pi$ is isomorphic to a fixed abelian surface.

As before, let $k$ be an algebraically closed field of characteristic $p>2$. 
Following \cite{Kudo}, let $C/k$ be a curve in $\mathbb{P}^3=\Proj (k[x,y,z,w])$ defined by $v=h=0$ for homogeneous polynomials $v,h\in k[x,y,z,w]$ with $\gcd(v,h)=1$. 
Let $r$ and $s$ denote the degrees of $v$ and $h$ respectively. Let $C^p$ denote the curve in $\mathbb{P}^3$ defined by $v^p=h^p=0$. For $n\in\mathbb{Z}$, let $\mathcal{O}_{\mathbb{P}^3}(n)$ denote the $n$th tensor power of Serre's twisting sheaf.

\begin{lemma} \cite[Lemma 3.1.3]{Kudo} \label{lem:commdiag1}
The following diagram is commutative with exact rows, where the composite map $\mathrm{H}^1(C,\mathcal{O}_C)\to \mathrm{H}^1(C,\mathcal{O}_C)$ is induced by the Frobenius morphism on $C$ and the map $F_1$ is the Frobenius morphism on $\mathbb{P}^3$.

\[
\xymatrix{0\ar[r] & \mathrm{H}^1(C,\mathcal{O}_C)\ar[r]\ar[d] & \mathrm{H}^3(\mathbb{P}^3,\mathcal{O}_{\mathbb{P}^3}(-r-s))\ar[r]\ar[d]^{F_1^*} & \mathrm{H}^3(\mathbb{P}^3,\mathcal{O}_{\mathbb{P}^3}(-r))\oplus \mathrm{H}^3(\mathbb{P}^3,\mathcal{O}_{\mathbb{P}^3}(-s))\ar[d]\\
0\ar[r] & \mathrm{H}^1(C^p,\mathcal{O}_{C^p})\ar[r] \ar[d]& \mathrm{H}^3(\mathbb{P}^3,\mathcal{O}_{\mathbb{P}^3}((-r-s)p))\ar[r]\ar[d]^{(vh)^{p-1}} & \mathrm{H}^3(\mathbb{P}^3,\mathcal{O}_{\mathbb{P}^3}(-rp))\oplus \mathrm{H}^3(\mathbb{P}^3,\mathcal{O}_{\mathbb{P}^3}(-sp))\ar[d]\\
0\ar[r] & \mathrm{H}^1(C,\mathcal{O}_C)\ar[r] & \mathrm{H}^3(\mathbb{P}^3,\mathcal{O}_{\mathbb{P}^3}(-r-s))\ar[r] & \mathrm{H}^3(\mathbb{P}^3,\mathcal{O}_{\mathbb{P}^3}(-r))\oplus \mathrm{H}^3(\mathbb{P}^3,\mathcal{O}_{\mathbb{P}^3}(-s))
}
\]

\end{lemma}

\begin{proof}
This is an excerpt from the diagram immediately preceding \cite[Proposition 3.1.4]{Kudo}.
\end{proof}

For $t\in\mathbb{Z}_{>0}$, the $k$-vector space $\mathrm{H}^3(\mathbb{P}^3, \mathcal{O}_{\mathbb{P}^3}(-t))$ has basis 
\[\{x^{k_1}y^{k_2}z^{k_3}w^{k_4} :(k_1,k_2,k_3,k_4)\in(\mathbb{Z}_{<0})^4, k_1+k_2+k_3+k_4=-t\}.\]  

\begin{lemma}\label{cor:commdiag2}
Suppose that $r\leq 3$. Then the following diagram is commutative with exact rows, where the map $F$ is the Frobenius morphism on $C$ and the map $F_1$ is the Frobenius morphism on $\mathbb{P}^3$.

\begin{equation}\label{eq:commdiag2}
\xymatrix{0\ar[r] & \mathrm{H}^1(C,\mathcal{O}_C)\ar[r]\ar[d]_{F^*} & \mathrm{H}^3(\mathbb{P}^3,\mathcal{O}_{\mathbb{P}^3}(-r-s))\ar[r]^{v}\ar[d]^{(vh)^{p-1}F_1^*} & \mathrm{H}^3(\mathbb{P}^3,\mathcal{O}_{\mathbb{P}^3}(-s))\ar[d]\\
0\ar[r] & \mathrm{H}^1(C,\mathcal{O}_C)\ar[r] & \mathrm{H}^3(\mathbb{P}^3,\mathcal{O}_{\mathbb{P}^3}(-r-s))\ar[r]^{v} & \mathrm{H}^3(\mathbb{P}^3,\mathcal{O}_{\mathbb{P}^3}(-s))
}
\end{equation}
\end{lemma}

\begin{proof}
This follows immediately from Lemma \ref{lem:commdiag1} and the fact that $\mathrm{H}^3(\mathbb{P}^3, \mathcal{O}_{\mathbb{P}^3}(-r))=0$ for $r\leq 3$.
\end{proof}

For $i,j\in\mathbb{Z}$, set $\delta_{ij}=\begin{cases}
1 & \textrm{if } i=j\\
0 & \textrm{if } i\neq j.
\end{cases}.$

Set $t=5$ and let 
\[S_{-5}=\{(k_1,k_2,k_3,k_4)\in(\mathbb{Z}_{<0})^4, k_1+k_2+k_3+k_4=-5\}.\]
Write $S_{-5}=\{(k^{(i)}_1,k^{(i)}_2,k^{(i)}_3,k^{(i)}_4)\}_{1 \leq i \leq 4}$, where 
$k^{(i)}_j=-1-\delta_{ij}$ for $1\leq i,j\leq 4$. 


\begin{proposition}\label{prop:Hasse-Witt}
Let $v$ and $h$ be homogeneous polynomials in $k[x,y,z,w]$ with $\gcd(v,h)=1$.
Suppose that $r={\rm deg}(v)=1$ and $s={\rm deg}(h)=4$.
Write $v=a_1x+a_2y+a_3z+a_4w$ and fix $t$, with $1 \leq t \leq 4$, such that $a_t\neq 0$. 
Let $C/k$ be the curve in $\mathbb{P}^3=\Proj (k[x,y,z,w])$ defined by $v=h=0$. 

Write $(vh)^{p-1} = \sum {c_{i_1,i_2,i_3,i_4}x^{i_1}y^{i_2}z^{i_3}w^{i_4}}$.  
For $1\leq i,j\leq 4$, write 
\[\gamma_{i,j}=c_{p(1+\delta_{1j})-(1+\delta_{1i}),\  p(1+\delta_{2j})-(1+\delta_{2i}), \ p(1+\delta_{3j})-(1+\delta_{3i}), \  p(1+\delta_{4j})-(1+\delta_{4i})}.\] Then the Hasse-Witt matrix of $C$ is given by 
\[{\rm HW}_C=(a_t^{p-1}\gamma_{i,j}-a_j^pa_t^{-1}\gamma_{i,t})_{1\leq i,j\leq 4, i\neq t, j\neq t}.\]
\end{proposition}

Let $t=4$ and $a_4=1$ and write $a=a_1$, $b=a_2$ and $c=a_3$.
Then the matrix ${\rm HW}_C$ in 
Proposition \ref{prop:Hasse-Witt} equals
{\tiny {\begin{equation*}
\begin{pmatrix}c_{2p-2,p-1,p-1,p-1} - a^p c_{p-2,p-1,p-1,2p-1}& c_{p-2,2p-1,p-1,p-1} - b^p c_{p-2,p-1,p-1,2p-1} &
c_{p-2,p-1,2p-1,p-1} - c^p c_{p-2,p-1,p-1,2p-1}  \\
c_{2p-1,p-2,p-1,p-1} -a^p c_{p-1,p-2,p-1,2p-1} & c_{p-1,2p-2,p-1,p-1} -b^p c_{p-1,p-2,p-1,2p-1} &c_{p-1,p-2,2p-1,p-1} - c^p c_{p-1,p-2,p-1,2p-1} \\
c_{2p-1,p-1,p-2,p-1} - a^p c_{p-1,p-1,p-2,2p-1} & c_{p-1,2p-1,p-2,p-1} - b^p c_{p-1,p-1,p-2,2p-1} & c_{p-1,p-1,2p-2,p-1} - c^p c_{p-1,p-1,p-2,2p-1} \\
\end{pmatrix}.
\end{equation*}}
}

\begin{proof}
Consider the multiplication-by-$v$ map $[\times v]: \mathrm{H}^3(\mathbb{P}^3,\mathcal{O}_{\mathbb{P}^3}(-5))\to \mathrm{H}^3(\mathbb{P}^3,\mathcal{O}_{\mathbb{P}^3}(-4))$.
By Lemma \ref{cor:commdiag2}, computing the matrix of $F^*$ is equivalent to computing the matrix of $(vh)^{p-1}F_1^*$ on the kernel of $[\times v]$. 

First, we compute the matrix of $(vh)^{p-1}F_1^*$ on all of $\mathrm{H}^3(\mathbb{P}^3,\mathcal{O}_{\mathbb{P}^3}(-5))$. The $k$-vector space $\mathrm{H}^3(\mathbb{P}^3, \mathcal{O}_{\mathbb{P}^3}(-5))$
 is $4$-dimensional with basis 
 \[\{x^{k_1}y^{k_2}z^{k_3}w^{k_4} :(k_1,k_2,k_3,k_4)\in S_{-5}\}.\]
 Explicitly, a basis is given by
 \[e_1=x^{-2}y^{-1}z^{-1}w^{-1}, e_2=x^{-1}y^{-2}z^{-1}w^{-1}, e_3=x^{-1}y^{-1}z^{-2}w^{-1}, e_4=x^{-1}y^{-1}z^{-1}w^{-2}.\] 

 As in the proof of \cite[Proposition 3.1.4]{Kudo}, for each $j\in \{1, \dots, 4\}$, then
  \begin{eqnarray*}
 (vh)^{p-1}F_1^*(e_j) &=& (vh)^{p-1}F_1^*(x^{k^{(j)}_1}y^{k^{(j)}_2}z^{k^{(j)}_3}w^{k^{(j)}_4})\\
  &=& (vh)^{p-1}x^{pk^{(j)}_1}y^{pk^{(j)}_2}z^{pk^{(j)}_3}w^{pk^{(j)}_4}\\
  &=&\sum {c_{i_1,i_2,i_3,i_4}x^{i_1+pk^{(j)}_1}y^{i_2+pk^{(j)}_2}z^{i_3+pk^{(j)}_3}w^{i_4+pk^{(j)}_4}}\\
  &=&\sum_{i=1}^4{c_{k_1^{(i)}-pk_1^{(j)}, k_2^{(i)}-pk_2^{(j)}, k_3^{(i)}-pk_3^{(j)}, k_4^{(i)}-pk_4^{(j)}}x^{k^{(i)}_1} y^{k^{(i)}_2} z^{k^{(i)}_3} w^{k^{(i)}_4}}\\
  &=&\sum_{i=1}^4{\gamma_{i,j}e_i}.
\end{eqnarray*}
 Explicitly, the following $4$-by-$4$ matrix $H_0$ represents the map 
$(vh)^{p-1}F_1^*$ on $\mathrm{H}^3(\mathbb{P}^3,\mathcal{O}_{\mathbb{P}^3}(-5))$, with respect to the basis $e_1, e_2, e_3 , e_4$:
\begin{equation} \label{E4by4}
H_0=\begin{pmatrix}c_{2p-2,p-1,p-1,p-1} & c_{p-2,2p-1,p-1,p-1} &c_{p-2,p-1,2p-1,p-1} & c_{p-2,p-1,p-1,2p-1} \\
c_{2p-1,p-2,p-1,p-1} & c_{p-1,2p-2,p-1,p-1} &c_{p-1,p-2,2p-1,p-1} & c_{p-1,p-2,p-1,2p-1} \\
c_{2p-1,p-1,p-2,p-1} & c_{p-1,2p-1,p-2,p-1} & c_{p-1,p-1,2p-2,p-1}& c_{p-1,p-1,p-2,2p-1} \\
c_{2p-1,p-1,p-1,p-2} & c_{p-1,2p-1,p-1,p-2} & c_{p-1,p-1,2p-1,p-2}& c_{p-1,p-1,p-1,2p-2}
\end{pmatrix}.
\end{equation}

Now we calculate the $3$-by-$3$ matrix representing the restriction of $(vh)^{p-1}F_1^*$ to the kernel of $[\times v]$ on $\mathrm{H}^3(\mathbb{P}^3,\mathcal{O}_{\mathbb{P}^3}(-5))$. First, note that if $\ell\in \mathrm{H}^3(\mathbb{P}^3,\mathcal{O}_{\mathbb{P}^3}(-5))$ is in ${\rm Ker}([\times v])$, then $(vh)^{p-1}F_1^*(\ell)$ is also in ${\rm Ker}([\times v])$, by the commutativity of \eqref{eq:commdiag2}.

The $k$-vector space $\mathrm{H}^3(\mathbb{P}^3, \mathcal{O}_{\mathbb{P}^3}(-4))$
 is $1$-dimensional with basis element $\lambda = x^{-1}y^{-1}z^{-1}w^{-1}$. Note that $v\cdot e_i=a_i\lambda$.
 Thus ${\rm Ker}([\times v]) = \{ \sum_{i=1}^4 c_i e_i  \mid \sum_{i=1}^4 a_i c_i = 0\}$.
For $1\leq i,j \leq 4$, write $\beta^{(i)}_j=a_ie_j-a_je_i$. 
If $a_t\neq 0$, then ${\rm Ker}([\times v])$ has basis $\{\beta^{(t)}_{j}\}_{1\leq j\leq 4, j\neq t}$. 
It follows that
\begin{eqnarray}
\nonumber (vh)^{p-1}F_1^*(\beta^{(t)}_j)&=&a_t^p(vh)^{p-1}F_1^*(e_j)-a_j^p(vh)^{p-1}F_1^*(e_t)\\
\nonumber&=& a_t^p\sum_{i=1}^4{\gamma_{i,j}e_i}-a_j^p\sum_{i=1}^4{\gamma_{i,t}e_i}\\
&=& \sum_{i=1}^4{(a_t^p\gamma_{i,j}-a_j^p\gamma_{i,t})e_i}.\label{eq:actonbeta}
\end{eqnarray}

The commutativity of the diagram \eqref{eq:commdiag2} shows that $(vh)^{p-1}F_1^*(\beta^{(t)}_j)$ is in ${\rm Ker}([\times v])$.
Therefore, there are coefficients $\lambda_{i,j}\in k$ such that for $j\neq t$, 
\begin{eqnarray}
\nonumber (vh)^{p-1}F_1^*(\beta^{(t)}_j)&=&\sum_{1\leq i\leq 4, i\neq t}{\lambda_{i,j}\beta^{(t)}_i}\\
&=& \sum_{1\leq i\leq 4,i\neq t}{\lambda_{i,j}(a_te_i-a_i e_t)}.\label{eq:actonbeta2}
\end{eqnarray}
Comparing the coefficients of $e_i$ for $i\neq t$ in \eqref{eq:actonbeta} and \eqref{eq:actonbeta2}, we see that \[\lambda_{i,j}=a_t^{-1}(a_t^p\gamma_{i,j}-a_j^p\gamma_{i,t}).\]
This completes the proof of Proposition \ref{prop:Hasse-Witt}.
\end{proof}

\section{The fiber of the Prym map when $g=3$}\label{sec:backwards}

In Section~\ref{sec:forwards}, we used a description from \cite{Bruin2008} of an unramified double cover 
$\pi:Y\to X$ of a plane quartic curve $X$ and its Prym $P_\pi$ in terms of quadratic forms.
We then calculated the Hasse-Witt matrices of $X$ and $P_\pi$ and produced examples where $X$ and $P_\pi$ 
have specified $p$-ranks for small primes $p$. However, since the entries of the Hasse-Witt matrices are very complicated, 
it is not clear how to apply this method for arbitrarily large primes $p$. 

In the current section, we describe an alternative method 
in which we start with a smooth curve $Z$ of genus $2$ over $k$
and construct smooth plane quartic curves $X$ 
having an unramified double cover $\pi: Y \to X$ 
such that $P_\pi \simeq \Jac(Z)$. 
The advantage of this alternative method is that it allows us to prove an existence result for infinitely many primes.  
In particular, 
in Proposition~\ref{Pexist5mod6}, we prove that
if $p \equiv 5 \bmod 6$, then there exists a smooth curve $X$ defined over ${\overline{\mathbb F}}_p$ with
genus $3$ and $p$-rank $3$
having an unramified double cover $\pi: Y \longrightarrow X$ such that 
$P_{\pi}$ has $p$-rank $0$.

Here is an outline of the section.
In Section \ref{SVerra}, we review a result of Verra that describes the geometry of the fibers 
of the Prym map ${\mathcal R}_3 \to {\mathcal A}_2$.
In Section \ref{Sexplicitfiber}, we work with an explicit construction of the curves
represented by points in the irreducible $3$-dimensional component of the fiber above ${\rm Jac}(Z)$.  
These curves occur as the intersection 
in ${\mathbb P}^3$ of a plane and the Kummer surface of ${\rm Jac}(Z)$.
They are smooth plane curves $X$ having
an unramified double cover $\pi:Y \to X$ 
such that $P_\pi \simeq {\rm Jac}(Z)$.
In Section \ref{SCommAlg}, we describe the determinant of the Hasse-Witt matrix of $X$.
The main application when $p \equiv 5 \bmod 6$ is in Section \ref{Sexist5mod6}.

In Section \ref{Scommalgmore}, we
use commutative algebra to characterize when $X$ is non-ordinary (under conditions on the $p$-rank of $Z$).
In Section \ref{Sp3fiber}, we fix $p=3$ and 
apply the results of the section to a $1$-dimensional family of genus $2$ curves $Z$ with $3$-rank $0$. 
This allows us to deduce information about the locus of planes $V$ for which $X$ is non-ordinary
and the geometry of the corresponding moduli space ${\mathcal R}_3^{(2,0)}$.

\subsection{Review of work of Verra} \label{SVerra}

Let $A$ be a principally polarized abelian surface.
Let ${\mathcal A}_2$ be the moduli space of principally polarized abelian surfaces.
Let $s$ be the point of ${\mathcal A}_2$ representing $A$.
We would like to consider the fiber of the Prym map 
$Pr_3: {\mathcal R}_3 \to {\mathcal A}_2$ over $s$. 
(More precisely, let $\tilde{\mathcal{A}}_2$ denote the smooth toroidal compactification of ${\mathcal A}_2$ 
and $\bar{\mathcal R}_3$ the compactification of ${\mathcal R}_3$ and consider the 
fiber of $\bar{Pr}_3: \bar{{\mathcal R}}_3 \to \tilde{{\mathcal A}}_2$ over $s$.)

Following \cite[Section~2]{Verra}, let $\Theta \subset A$ be a symmetric theta divisor.
Suppose that ${\rm Aut}(\Theta) \simeq {\mathbb Z}/2$.
Under this mild condition on $s$, Verra proves in \cite[Corollary 4.1]{Verra} that 
$\bar{Pr}_3^{-1}(s)$ is a blow-up of ${\mathbb P}^3$.  Moreover, $\bar{Pr}_3^{-1}(s)$
has one irreducible component $N_s$ of dimension $3$ 
and three components of dimension $2$.  By \cite[(3.14)-(3.16), page 442]{Verra}, 
the latter represent unramified double covers of hyperelliptic or singular curves whose Prym is isomorphic to $A$. 
The generic point of $N_s$ represents an unramified double cover $\pi:Y \to X$ where $X$ is a 
smooth plane quartic and $P_\pi \simeq A$.
We briefly review the results of Verra in more detail below.

The linear system $|2 \Theta|$ has dimension $3$ and is thus isomorphic to ${\mathbb P}^3$.
Every element $\tilde{C}$ in the linear system is a curve of arithmetic genus $5$ with an involution.
The linear system is base point free and its generic element is a smooth irreducible curve.
For each $\tilde{C} \in |2 \Theta|$, there is a morphism $\psi: A \to ({\mathbb P}^3)^{\wedge}$, 
where the wedge in superscript indicates taking the dual space.
Let $K=\psi(A)$. 
Then ${\rm deg}(\psi) = 2$ if and only if $A={\rm Jac}(Z)$ for some smooth irreducible curve $Z$ of genus $2$; 
(if not, then ${\rm deg}(\psi)=4$).
If ${\rm deg}(\psi)=2$, then $K$ is the Kummer quartic surface of $A$. 

By \cite[page 438]{Verra}, this yields a map $\phi: {\mathbb P}^3 - T \to \bar{Pr}_3^{-1}(s)$.
Here $T$ denotes the set of $\tilde{C}$ which are not stable.
By \cite[(2.1)]{Verra}, $T= B \cap K^{\wedge}$.  Note that
$K^{\wedge} \subset ((\mathbb{P}^3)^{\wedge})^{\wedge} \simeq {\mathbb P}^3$ is birational to $K$.
Here $B$ denotes the union of $B_\tau$ where $\tau$ is a 2-torsion point of $A$ and 
$B_\tau$ denotes the set of $\tilde{C}$ in the linear system $|2 \Theta|$ such that $\tilde{C}$ contains $\tau$.

\subsection{Explicit version of the fiber of the Prym map} \label{Sexplicitfiber}

Let $Z$ be a smooth curve of genus 2. 
The results of \cite{Verra} give a way to find all smooth plane quartics $X$ having an unramified double cover 
$\pi:Y \to X$ such that $P_\pi \simeq {\rm Jac}(Z)$.
This is discussed in \cite[Section 7]{Bruin2008}, where Bruin shows how to recover a model 
of the form $X: Q_1Q_3=(Q_2)^2$ from a smooth plane section $X$ of the Kummer surface $K$.

Consider the Kummer surface $$K = \Jac(Z)/ \langle -1 \rangle \subset {\mathbb P}^3$$ associated to $Z$, 
namely the quotient of $\Jac(Z)$ by the Kummer involution. 
It is a quartic surface, with 16 singularities corresponding to $\Jac(Z) [2] \simeq (\mathbb{Z} / 2 \mathbb{Z})^{2g}$.
Let $\phi: \Jac(Z) \longrightarrow K$ be the degree $2$ quotient map.

For a sufficiently general plane $V \subseteq {\mathbb P}^3$, the intersection $$X = V \cap K$$ is a smooth quartic plane curve. 
This implies that $X$ does not contain any of the branch points of $\phi$.
Thus the restriction of $\phi$ to $Y = \phi^{-1}(X)$ is an unramified double cover $\pi:Y \to X$.
Since $Y$ is in $|2 \Theta|$, 
the Prym variety $P_\pi$ is isomorphic to $\Jac(Z)$, as seen on \cite[page 616]{MR1013156}.
Conversely, by Verra's result, if $\Jac(Z)$ is isomorphic to the Prym variety of 
an unramified double cover $\pi:Y \to X$, with $X$ a smooth plane quartic, 
then $X$ is isomorphic to a planar section of $K$ and $Y$ is its preimage in $\Jac(Z)$.

\subsubsection{The Kummer surface} \label{kummer}

Suppose that $Z$ is a smooth curve of genus $2$ with affine equation $Z: z^2=D(x):=\sum_{i=0}^6 d_ix^i$.
Consider the Kummer surface $K = \Jac(Z)/ \langle-1\rangle$ associated to $Z$, 
which is a quartic surface in $\mathbb{P}^3$. 
In this section, we write down the equation of $K$ as found in \cite[Chapter~3, Section~1]{CasselsFlynn}.

There is a map $\phi: \Jac(Z) \to K$ defined as follows.
A generic divisor of degree $2$ on $Z$ has the form $(x_1,z_1) + (x_2,z_2)$. 
Let $Z_\infty$ be the divisor above $x=\infty$.  Then
\begin{align*}
  \phi \colon \Jac(Z) &\rightarrow K \\
  [(x_1,z_1)+(x_2,z_2) -Z_\infty] &\mapsto [1:x_1+x_2:x_1x_2:\beta_0],
\end{align*}
where $\beta_0=(F_0(x_1,x_2)-2z_1z_2)/(x_1-x_2)^2$ and $F_0(x_1,x_2)$ equals
\[2d_0+d_1(x_1+x_2)+2d_2x_1x_2+d_3(x_1+x_2)x_1x_2+2d_4(x_1x_2)^2+d_5(x_1+x_2)(x_1x_2)^2+2d_6(x_1x_2)^3.\] 
The map $\phi$ realizes $\Jac(Z)$ as a double cover of $K$ that ramifies precisely at $\Jac(Z)[2]$. 
It maps the $16$ points of order $2$ of $\Jac(Z)$ to the 16 singularities of $K$. 

 Let $X_1, \ldots, X_4$ denote the coordinates on $\mathbb{P}^3$.
 By \cite[(3.1.8)]{CasselsFlynn}, a projective model of the Kummer surface $K$ in ${\mathbb P}^3$ is the zero locus of the following equation, 
\begin{eqnarray} \label{Ekummer}
\kappa(X_1,X_2,X_3,X_4)=K_2X_4^2+K_1X_4+K_0
\end{eqnarray}
with
\begin{eqnarray*}
K_2&=&X_2^2-4X_1X_3\\
K_1&=&-2(2d_0X_1^3+d_1X_1^2X_2+2d_2X_1^2X_3+d_3X_1X_2X_3+2d_4X_1X_3^2+d_5X_2X_3^2+2d_6X_3^3)\\
K_0&=&(d_1^2-4d_0d_2)X_1^4-4d_0d_3X_1^3X_2-2d_1d_3X_1^3X_3-4d_0d_4X_1^2X_2^2\\
&&+4(d_0d_5-d_1d_4)X_1^2X_2X_3+(d_3^2+2d_1d_5-4d_2d_4-4d_0d_6)X_1^2X_3^2-4d_0d_5X_1X_2^3\\
&&+4(2d_0d_6-d_1d_5)X_1X_2^2X_3+4(d_1d_6-d_2d_5)X_1X_2X_3^2-2d_3d_5X_1X_3^3-4d_0d_6X_2^4\\
&&-4d_1d_6X_2^3X_3-4d_2d_6X_2^2X_3^2-4d_3d_6X_2X_3^3+(d_5^2-4d_4d_6)X_3^4,
\end{eqnarray*}
where $d_0, \ldots, d_6$ are the coefficients of $D(x)$ in the equation for $Z$. 

This model of the Kummer surface $K$ in ${\mathbb P}^3$ arises from a projective model of $\Jac(Z)$ in ${\mathbb P}^{15}$;
explicit calculations are thus more efficient on the Kummer surface.
Alternatively, by combining a few Frobenius identities of theta characteristic functions, 
one can derive another projective model of $K$ parametrized by four theta constants \cite[Section 7, Chapter IIIa]{Mumford}.

\subsubsection{Plane quartics as planar sections of the Kummer surface} \label{Splanequartic}

Let $K$ be the Kummer surface from \eqref{Ekummer}.
For a plane $V \subset {\mathbb P}^3$,
consider the curve
$$ X = V \cap K.$$  If $X$ is smooth, then it has genus $3$ and the pullback of ${\rm Jac}(Z) \to K$ to $X$ yields an 
unramified double cover $\pi:Y \to X$ such that the Prym variety $P_\pi$ is isomorphic to $\Jac(Z)$. 

Let $V=V_{a,b,c,d}$ be a plane defined over $k$ by 
\begin{equation} \label{DefV}
V : \ v(X_1,X_2,X_3,X_4)=aX_1+bX_2+cX_3+dX_4=0.
\end{equation}
The point $(0:0:0:1)$ is a singular point of the Kummer surface.  For planes $V$ which avoid the singularities of $K$, 
it is no restriction to take $d=1$.

\subsection{The Hasse-Witt matrix of $X$} \label{SCommAlg}

In Section \ref{sec:Frobenius}, we determined the Hasse-Witt matrix for a curve $X$ given 
as the intersection of a plane and quartic surface in ${\mathbb P}^3$. 
Recall that $X = V \cap K \subset {\mathbb P}^3$ where $K: \kappa =0$ and $V: v=0$ 
are defined in \eqref{Ekummer} and \eqref{DefV}.
Recall that $c_{i_1,i_2,i_3,i_4}$ is the coefficient of $X_1^{i_1}X_2^{i_2}X_3^{i_3}X_4^{i_4}$ in $(v \kappa)^{p-1}$.
In other words, 
$$(v \kappa)^{p-1}= \sum_{i_1+i_2+i_3+i_4=5(p-1)} c_{i_1,i_2,i_3,i_4} X_1^{i_1}X_2^{i_2}X_3^{i_3}X_4^{i_4}.$$ 
Proposition~\ref{prop:Hasse-Witt} describes the Hasse-Witt matrix $H_X$ of $X$ in terms of the coefficients $c_{i_1,i_2,i_3,i_4}$ and $a,b,c,d$. 

\begin{lemma} \label{LhomHX}
Setting $d=1$, then the coefficients of $H_X$ are each homogeneous of degree $2(p-1)$ in 
$a,b,c, d_0, \ldots, d_6$.
\end{lemma}

\begin{proof}
First, note that the equation $\kappa$ in \eqref{Ekummer} for $K$ is homogeneous of degree $2$ in $d_0, \ldots, d_6,X_4$.
This is because $K_2 X_4^2$, $K_1 X_4$, and $K_0$ are each homogeneous of degree $2$ 
in $d_0, \ldots, d_6,X_4$.
Also, the equation $v$ for $V$ is homogeneous of degree $1$ in $a,b,c,X_4$.
Thus $(\kappa v)^{p-1}$ is 
homogeneous of degree $3(p-1)$ in $a,b,c, d_0, \ldots, d_6, X_4$.
The coefficients of the $4 \times 4$ matrix $H_0$ from \eqref{E4by4} are coefficients of $(\kappa v)^{p-1}$. 

We now determine information about the coefficients of the Hasse-Witt matrix $H_X$.
Set $d=1$.  Let $U$ be the $3 \times 3$ matrix obtained by removing the $4$th row and $4$th column of $H_0$.
Let $C = [c_{p-2,p-1,p-1,2p-1}, c_{p-1, p-2, p-1, 2p-1}, c_{p-1,p-1,p-2,2p-1}]^T$.
By Proposition \ref{prop:Hasse-Witt}, $H_X = U - C [a^p, b^p, c^p]$.
The coefficients of $U$ are of the form $c_{i_1,i_2,i_3,i_4}$ with $i_4 = p-1$;
thus they are each homogeneous of degree $2(p-1)$ in $a,b,c, d_0, \ldots, d_6$. 
The coefficients of $C$ are of the form $c_{i_1,i_2,i_3,i_4}$ with $i_4 = 2p-1$;
thus they are each homogenous of degree $p-2$ in $a,b,c, d_0, \ldots, d_6$. 
Thus each coefficient of $H_X = U - C [a^p, b^p, c^p]$ is homogenous of degree $2(p-1)$ in 
$a,b,c, d_0, \ldots, d_6$.
\end{proof}

\subsection{An existence result for each $p \equiv 5 \bmod 6$} \label{Sexist5mod6}

\begin{proposition} \label{Pexist5mod6}
If $p \equiv 5 \bmod 6$, then there exists a smooth curve $X$ defined over ${\overline{\mathbb F}}_p$ with
genus $3$ and $p$-rank $3$
having an unramified double cover $\pi: Y \longrightarrow X$ such that 
$P_{\pi}$ has $p$-rank $0$.
More generally, $R_3^{(3,0)}$ is non-empty of dimension $4$
for each prime $p \equiv 5 \bmod 6$. 
\end{proposition}

\begin{proof}
Consider the genus $2$ curve $Z: z^2=x^6-1$; 
it is superspecial, and thus has $p$-rank $0$, when $p \equiv 5 \bmod 6$ \cite[Proposition 1.11]{IKO}.
The Kummer surface $K$ in ${\mathbb P}^3$ is the zero locus of
\begin{eqnarray} 
\kappa(X_1,X_2,X_3,X_4)=K_2X_4^2+K_1X_4+K_0
\end{eqnarray}
with
\begin{eqnarray*}
K_2&=&X_2^2-4X_1X_3\\
K_1&=&4X_1^3-4X_3^3\\
K_0&=& 4 X_1^2X_3^2- 8 X_1X_2^2X_3+4 X_2^4.
\end{eqnarray*}

So 
\[\kappa=X_2^2X_4^2 - 4 X_1X_3X_4^2 + 4X_1^3X_4 - 4 X_3^3X_4 + 4 X_1^2X_3^2 - 8X_1X_2^2X_3 + 4 X_2^4.\]

Let $v=aX_1 + bX_2+cX_3 + X_4$ (so $d=1$).
Let $c_{i_1,i_2,i_3,i_4}$ be the coefficient of $X_1^{i_1}X_2^{i_2}X_3^{i_3}X_4^{i_4}$ in 
$(\kappa v)^{p-1}$.
By Proposition~\ref{prop:Hasse-Witt},
the Hasse-Witt matrix $H_X$ of $X=V \cap K$ is 
{\tiny {\begin{equation*}
\begin{pmatrix}c_{2p-2,p-1,p-1,p-1} - a^p c_{p-2,p-1,p-1,2p-1}& c_{p-2,2p-1,p-1,p-1} - b^p c_{p-2,p-1,p-1,2p-1} &
c_{p-2,p-1,2p-1,p-1} - c^p c_{p-2,p-1,p-1,2p-1}  \\
c_{2p-1,p-2,p-1,p-1} -a^p c_{p-1,p-2,p-1,2p-1} & c_{p-1,2p-2,p-1,p-1} -b^p c_{p-1,p-2,p-1,2p-1} &c_{p-1,p-2,2p-1,p-1} - c^p c_{p-1,p-2,p-1,2p-1} \\
c_{2p-1,p-1,p-2,p-1} - a^p c_{p-1,p-1,p-2,2p-1} & c_{p-1,2p-1,p-2,p-1} - b^p c_{p-1,p-1,p-2,2p-1} & c_{p-1,p-1,2p-2,p-1} - c^p c_{p-1,p-1,p-2,2p-1} \\
\end{pmatrix}.
\end{equation*}}
}

By Lemma~\ref{lem:Pexist5mod6} (below), 
when $p\equiv 5\bmod 6$, then the determinant of $H_X$ has degree $4(p-1)$ when 
considered as a polynomial in $b$.
In particular, ${\rm det}(H_X)$ is a non-zero polynomial in $a,b,c$.
The condition that $X$ is singular is a non-zero polynomial condition in $a,b,c$.
Therefore, there exists a triple $(a,b,c)\in \overline{\mathbb{F}}_p^3$ such that $X$ is smooth and ${\rm det}(H_X) \not = 0$.
This implies that $X$ is ordinary, with $p$-rank $3$, and the unramified double cover $\pi:Y \to X$ has the property 
that $P_{\pi} \simeq {\rm Jac}(Z)$ has $p$-rank $0$.
Thus $R_3^{(3,0)}$ is non-empty. 
The dimension result follows from \cite[Proposition 5.2]{OzmanPries2016}.
\end{proof}

We remark that a result similar to Proposition~\ref{Pexist5mod6} may be true when 
$p \equiv 5,7 \bmod 8$ with $Z:z^2=x^5-x$ or when $p \equiv 2,3,4 \bmod 5$ with $Z:z^2=x^5-1$. 

The next lemma provides the cornerstone of the proof of Proposition~\ref{Pexist5mod6}.

\begin{lemma}\label{lem:Pexist5mod6}
Let $p\equiv 5\bmod 6$ and let $X$ be as in the proof of Proposition~\ref{Pexist5mod6}.
When considered as a polynomial in $b$, the determinant of $H_X$ has degree $4(p-1)$.
\end{lemma}

\begin{proof}
When considered as a polynomial in $b$, the coefficient $c_{i_1,i_2,i_3,i_4}$ of $X_1^{i_1}X_2^{i_2}X_3^{i_3}X_4^{i_4}$ in 
$(\kappa v)^{p-1}$ has degree at most $p-1$. 
Any occurrence of $b$ comes from the term $bX_2$ in $v$, so $c_{i_1,i_2,i_3,i_4}$ has degree at most $i_2$ in $b$.

Note that $\kappa$ has degree $2$ in $X_4$, so $\kappa^{p-1}$ has degree $2p-2$ in $X_4$.
Any monomial in $\kappa^{p-1}$ not divisible by $X_2$ arises as the product 
\begin{equation}\label{Emformula}
(-4X_1X_3X_4^2)^{m_1}(4X_1^3X_4)^{m_2}(-4X_3^3X_4)^{m_3}(4X_1^2X_3^2)^{m_4}
\end{equation} for some $m_1,m_2,m_3,m_4\in\mathbb{Z}^{\geq 0}$ with $m_1+m_2+m_3+m_4=p-1$. 

\smallskip

\paragraph{\textbf{Claim 1:}} When considered as a polynomial in $b$, any term of the form $c_{i_1,p-1,i_3,2p-1}$ has degree at most $p-2$.

\smallskip

\paragraph{\bf{Proof of Claim 1:} } Any occurrence of $b^{p-1}$ in $(\kappa v)^{p-1}$ comes from $\kappa^{p-1}(bX_2)^{p-1}$. The coefficient of $X_1^{i_1}X_3^{i_3}X_4^{2p-1}$ in 
$\kappa^{p-1}$ is zero because $\kappa^{p-1}$ has degree $2p-2$ in $X_4$. $\Box$

\smallskip

By Claim 1, in the middle column of $H_X$, the top and bottom entries,
\[c_{p-2,2p-1,p-1,p-1} - b^p c_{p-2,p-1,p-1,2p-1} {\rm \ and \ } c_{p-1,2p-1,p-2,p-1} - b^p c_{p-1,p-1,p-2,2p-1},\] have degrees at most $2p-2$ in $b$. We consider the six terms in the expansion of ${\rm det}(H_X)$. The four terms that do not contain the central coefficient of $H_X$ have degrees at most $2p-2+p-1+p-2=4p-5$. It remains to consider the product of the diagonal coefficients, and the product of the antidiagonal coefficients. We show that the former has degree at most $4p-6$ and the latter has degree $4p-4$ as polynomials in $b$.

\smallskip

\paragraph{\textbf{Claim 2:}} When considered as a polynomial in $b$, each of the two terms $c_{k(p-1),p-1,\ell(p-1),p-1}$ for $(k,\ell)\in \{(1,2),(2,1)\}$ has degree at most $p-2$.

\smallskip

\paragraph{\bf{Proof of Claim 2:} } 
We must show that the coefficient of $X_1^{k(p-1)}X_3^{\ell(p-1)}X_4^{p-1}$ in 
$\kappa^{p-1}$ is zero. Since $X_2$ does not divide this monomial, it appears as a product as in \eqref{Emformula}.
In order to attain the desired powers of $X_1$ and $X_3$, we must have 
\[m_1+3m_2+2m_4=k(p-1)\]
and 
\[m_1+3m_3+2m_4=\ell(p-1).\]
Subtracting the two equations gives $3(m_2-m_3)=\pm (p-1)$. But $3\nmid (p-1)$ since $p\equiv 5\bmod 6$. So the coefficient of $X_1^{k(p-1)}X_3^{\ell(p-1)}X_4^{p-1}$ in 
$\kappa^{p-1}$ is zero, as required.$\Box$

\smallskip

Combining Claims 1 and 2 shows that the product of the diagonal entries of $H_X$ has degree at most $p-2+2p-2+p-2=4p-6$ in $b$. Finally, we show that the product of the antidiagonal entries has degree $4p-4$ when considered as a polynomial in $b$.

\smallskip

\paragraph{\textbf{Claim 3:}} When considered as polynomials in $b$, the following terms have degree $p-1$:
\[(1) \ c_{p-2,p-1,2p-1,p-1} {\rm \ and \ } (2) \ c_{2p-1,p-1,p-2,p-1}.\]

\smallskip

\paragraph{\bf{Proof of Claim 3:} } 
For (1), 
any occurrence of $b^{p-1}$ in $(\kappa v)^{p-1}$ comes from $\kappa^{p-1}(bX_2)^{p-1}$. 
Hence we must show that the coefficient of $X_1^{p-2}X_3^{2p-1}X_4^{p-1}$ in 
$\kappa^{p-1}$ is nonzero. Since $X_2$ does not divide this monomial, it arises as a product as in \eqref{Emformula}.
In order to attain the desired power of $X_4$, we must have 
\[2m_1+m_2+m_3=p-1,\]
whereby $m_1=m_4$.
In order to attain the desired powers of $X_1$ and $X_3$, we must have
\begin{eqnarray*}
m_1+3m_2+2m_4=3(m_1+m_2)&=&p-2\\
\textrm{and} \\
m_1+3m_3+2m_4=3(m_1+m_3)&=&2p-1.
\end{eqnarray*}
So $m_1$ determines $m_2,m_3,m_4$. Write $p=6k+5$ for some integer $k$, so $m_2=2k+1-m_1$ and $m_3=4k+3-m_1$.
The coefficient of $b^{p-1}X_1^{p-2}X_2^{p-1}X_3^{2p-1}X_4^{p-1}$ in $c_{p-2,p-1,2p-1,p-1}$ is
\begin{align*}
&\sum_{m_1+m_2+m_3+m_4=p-1} \frac{(p-1)!}{m_1!m_2!m_3!m_4!}(-4)^{m_1}4^{2k+1-m_1} (-4)^{4k+3-m_1}4^{m_1} \\
&=(-1)^{4k+3} 4^{p-1}\sum_{m_1+m_2+m_3+m_4=p-1}\frac{(p-1)!}{m_1!m_2!m_3!m_4!} .\end{align*}

Note that
\begin{align*}
\sum_{m_1+m_2+m_3+m_4=p-1}\frac{(p-1)!}{m_1!m_2!m_3!m_4!}=4^{p-1}.
\end{align*}

Therefore, the coefficient of $b^{p-1}X_1^{p-2}X_2^{p-1}X_3^{2p-1}X_4^{p-1}$ is 
the nonzero number \[(-1)^{4k+3}4^{p-1}4^{p-1}=-4^{2p-2}.\]

For (2), consider the coefficient of $b^{p-1}$ in $c_{2p-1,p-1,p-2,p-1}$.  By the same strategy as above:
\begin{align*}
m_1=m_4;\\
  3m_1+3m_2=2p-1=12k+9;\\
   3m_1+3m_3= p-2=6k+3.
\end{align*}
So, the coefficient of $b^{p-1}X_1^{2p-1}X_2^{p-1}X_3^{p-2}X_4^{p-1}$ in $c_{2p-1,p-1,p-2,p-1}$ is the nonzero number

\begin{align*}
&\sum_{m_1+m_2+m_3+m_4=p-1} \frac{(p-1)!}{m_1!m_2!m_3!m_4!}(-4)^{m_1}4^{4k+3-m_1} (-4)^{2k+1-m_1}4^{m_1} \\
&=(-1)^{2k+1} 4^{p-1}\sum_{m_1+m_2+m_3+m_4=p-1}\frac{(p-1)!}{m_1!m_2!m_3!m_4!}=-4^{2p-2}.  \Box
\end{align*}

\smallskip

\paragraph{\textbf{Claim 4:}} When considered as a polynomial in $b$, the term $c_{p-1,p-2,p-1,2p-1}$ has degree $p-2$.

\smallskip

\paragraph{\bf{Proof of Claim 4:} } Consider the coefficient of $b^{p-2}$ in $c_{p-1,p-2,p-1,2p-1}$. 
Since $\kappa^{p-1}$ has degree $2p-2$ in $X_4$, any occurrence of $b^{p-2}X_1^{p-1}X_2^{p-2}X_3^{p-1}X_4^{2p-1}$ in $(\kappa v)^{p-1}$ comes from choosing the monomial $bX_2$ in $p-2$ factors $v$ of $v^{p-1}$, and $X_4$ in the remaining one.
 There are $p-1$ ways of doing so. 

 Now we compute the coefficient of $X_1^{p-1}X_3^{p-1}X_4^{2p-2}$ in $\kappa^{p-1}$. A monomial divisible by $X_2$ cannot be chosen in a factor $\kappa$ of $\kappa^{p-1}$. Therefore, to obtain the exponent $2p-2$ of $X_4$, we need to pick the monomial $ - 4 X_1X_3X_4^2$ in each factor $\kappa$ of $\kappa^{p-1}$. Hence, the coefficient of $b^{p-2}X_1^{p-1}X_2^{p-2}X_3^{p-1}X_4^{2p-1}$ in $c_{p-1,p-2,p-1,2p-1}$ is $(p-1)(-4)^{p-1}$, which is not zero.
$\Box$.

By the claims, the leading power of $b$ arises in the product of the antidiagonal entries and has degree $p-1+p+p-2+p-1=4p-4$.
\end{proof}

\subsection{The condition that $X$ is non-ordinary} \label{Scommalgmore}

In this subsection, $Z$ is an arbitrary smooth curve of genus $2$ and 
$X = K \cap V$ with Hasse-Witt matrix $H_X$ as in Section \ref{SCommAlg}.

\subsubsection{The condition that $X$ is non-ordinary}

The curve $X$ is non-ordinary if and only if ${\rm det}(H_X)$ is non-zero.

\begin{proposition} \label{PhomogX}
Setting $d=1$, then ${\rm det}(H_X)$ is non-zero and homogeneous of degree $6(p-1)$ in 
$a,b,c, d_0, \ldots, d_6$.
\end{proposition}

\begin{proof}
By Lemma \ref{LhomHX}, each coefficient of the $3 \times 3$ matrix $H_X$ is homogeneous of 
degree $2(p-1)$ in these variables.  Thus it suffices to prove that ${\rm det}(H_X)$ is non-zero.  
We expect that this can be proven algebraically, but to avoid long computations we continue with the following 
theoretical argument.
If ${\rm det}(H_X)$ is identically $0$, then a generic point of $\mathcal{R}_3 = \Pi^{-1}(\mathcal{M}_3)$ 
would represent an unramified double cover $\pi: Y \to X$ such that $X$ is non-ordinary; this is false.
\end{proof}

We apply a fractional linear transformation to $x$ in order to reduce the number of variables defining $Z$, 
while preserving the degree of ${\rm det}(H_X)$ and its homogeneous property.
Without loss of generality, we can suppose that no Weierstrass point of $Z$ lies over $x = \infty$
and that 3 of the Weierstrass points are rational and 
lie over $x=0,1,-1$; 
(over a non-algebraically closed field, this may only be true after a finite extension).
Then,
\begin{eqnarray} \label{EsimpleZ}
Z:z^2 & = & D(x):= (x^3-x)(A_0 x^3+A x^2 + B x + C)\\
& = & A_0 x^6 + A x^5 + (B-A_0) x^4 + (C - A) x^3 - B x^2 - C x.
\end{eqnarray} 
Note that $A_0 \not = 0$ by the hypothesis at $\infty$ and $C \not = 0$ since $Z$ is smooth.

\subsubsection{The condition that $X$ is non-ordinary and $Z$ has $p$-rank $1$}

As in \eqref{EsimpleZ}, write $Z:z^2 =  D(x):= (x^3-x)(A_0 x^3+A x^2 + B x + C)$.
The curve $Z$ is not ordinary if and only if ${\rm det}(H_Z)=0$

\begin{lemma} \label{LnonordZ}
Then ${\rm det}(H_X)$ does not vanish identically under the polynomial condition ${\rm det}(H_Z)=0$, which is homogenous of degree $p-1$ in $A_0, A, B, C$. 
\end{lemma}

\begin{proof}
Since $D(x)$ is linear in $A_0, A, B, C$, the entries of the Hasse-Witt matrix $H_Z$ are homogenous of degree $(p-1)/2$ so ${\rm det}(H_Z)$ 
is homogenous of degree $p-1$. 

For the non-vanishing claim, it suffices to show that $X$ is typically ordinary when $Z$ has $p$-rank $1$.  This follows from the fact that each component of $\mathcal{R}_3^{(3,1)}$ has dimension $5$, while each component of $\mathcal{R}_3^{(2,1)}$ has dimension $4$, which is a special case of \cite[Theorem 7.1]{OzmanPries2016}.
\end{proof}

For example, when $p=3$ and $(A_0,A,B,C,a,b,c)=(1, 0, 1, 2t, 0, 1, 1)$ with $t \in {\mathbb F}_9$ a root of $t^2-t-1$
then $f=1$ and $f'=1$ and the curve $X$ is smooth.

\subsubsection{The condition that $X$ is non-ordinary and $Z$ has $p$-rank $0$}

\begin{lemma} \label{LhomogZ}
The curve $Z$ is not ordinary under a polynomial condition on $A_0, A, B, C$ which is homogenous of degree
$p-1$.  The curve $Z$ has $p$-rank $0$ under 4 polynomial conditions on 
$A_0, A, B, C$ which are each homogenous of degree $(p+1)(p-1)/2$.  
\end{lemma}

\begin{proof}
The curve $Z$ has $p$-rank $0$ if and only if $H_Z H_Z^{(p)} =[0]$.  The entries of $H_Z^{(p)}$ are homogenous of 
degree $p(p-1)/2$, so the entries of $H_Z H_Z^{(p)}$ are homogenous of degree $(p+1)(p-1)/2$.
\end{proof}

\begin{proposition} \label{Phomogeneous}
Let $Z$ be a genus $2$ curve with equation $z^2=(x^3-x)(A_0 x^3+A x^2 + B x + C)$.
Let $V$ be a plane with equation $aX_1 + bX_2 + cX_3 + X_4$.
Let $X = V \cap K$ and consider the unramified double cover $\pi:Y \to X$ 
given by the restriction of $\phi:{\rm Jac}(Z) \to K$.
Then the condition that $X$ is non-ordinary and $Z$ has $p$-rank $0$ 
is given by the vanishing of 4 homogeneous polynomials of degree $(p+1)(p-1)/2$ in $A_0, A, B, C$
and the vanishing of one homogeneous polynomial ${\rm det}(H_X)$ of degree $6(p-1)$ in $a,b,c,A_0,A,B,C$.
\end{proposition}

\begin{proof}
The curve $X$ is non-ordinary if and only if ${\rm det}(H_X)$ vanishes.
By Proposition \ref{PhomogX}, ${\rm det}(H_X)$ is homogenous of degree $6(p-1)$ in $a,b,c$ and the coefficients of $D(x)$, 
which are linear in $A_0,A,B,C$.
The curve $Z$ has $p$-rank $0$ under the conditions in Lemma \ref{LhomogZ}.
\end{proof}

We expect that the answer to Question \ref{Qnatural} is yes for all odd $p$ when $g=3$, $f=2,3$, and $f'=0$.
We now rephrase the question in those cases to a question in commutative algebra.

\begin{question}
For all odd primes $p$,
is there a plane $V$ for which ${\rm det}(H_X)$ does not vanish identically under the 4 conditions when $H_Z  H_Z^{(p)} =[0]$?
Is there a plane $V$ for which ${\rm det}(H_X)$ does vanish for some $Z$ such that $H_Z  H_Z^{(p)} =[0]$?
\end{question}

The difficulty in showing that ${\rm det}(H_X)$ does not vanish identically when $Z$ has $p$-rank $0$ is that we do not know much about the variety of the ideal generated by the 4 polynomial conditions when $H_Z  H_Z^{(p)} =[0]$ or the behavior of the derivatives of ${\rm det}(H_X)$ with respect to the variables $a,b,c$.

\subsubsection{Example: when $p=3$}

Write $Z:z^2 = D(x)= (x^3-x)(A_0 x^3+A x^2 + B x + C)$.
Then 
\[H_Z=\left( \begin{array}{cc}
-B & A \\
-C & B- A_0
\end{array} \right).\]
The 4 entries of $H_Z H_Z^{(3)}$ are:
\[B^4 - AC^3, \ C(B^3 -C^2(B-A_0)), \ A((B-A_0)^3-BA^2), -CA^3+(B-A_0)^4.\]

Recall that $C \not = 0$ since $Z$ is smooth.  
If $Z$ has $3$-rank $0$ then if any of $A,B,B-A_0$ are zero then all of them are zero, which implies $A_0=0$, which contradicts the hypothesis at $\infty$.  
Thus $AB(B-A_0) \not = 0$ when $Z$ has $3$-rank $0$.
One can check that $H_Z H_Z^{(3)}=[0]$ if and only if
\[B^4 - AC^3=0, B^3 -C^2(B-A_0)=0.\]

Write 
$$(v\kappa)^2= \sum_{i_1+i_2+i_3+i_4=10} c_{i_1,i_2,i_3,i_4} X_1^{i_1}X_2^{i_2}X_3^{i_3}X_4^{i_4},$$ 
and assume that $d=1$.  Then the Hasse-Witt matrix $H_X$ is given by
\begin{eqnarray}\label{HW3} \left( \begin{array}{ccc}
c_{4,2,2,2}-a^3 c_{1,2,2,5} & c_{1,5,2,2}-b^3 c_{1,2,2,5}  & c_{1,2,5,2}- c^3 c_{1,2,2,5} \\
c_{5,1,2,2}-a^3 c_{2,1,2,5} & c_{2,4,2,2}-b^3 c_{2,1,2,5} & c_{2,1,5,2}-c^3 c_{2,1,2,5} \\
c_{5,2,1,2}-a^3 c_{2,2,1,5} & c_{2,5,1,2}-b^3 c_{2,2,1,5} & c_{2,2,4,2}- c^3 c_{2,2,1,5} \end{array} \right)
\end{eqnarray}
By Proposition~\ref{prop:prank}, the $3$-rank of $X$ is the stable rank of $H_X$, 
which is the rank of $H_X H_X^{(3)} H_X^{(3^2)}$.
The entries of $H_X$ are homogenous of degree $4$
and ${\rm det}(H_X)$ is homogenous of degree $12$ in $a, b, c, A_0, A, B, C$.

\subsubsection{Example: Genus $3$ curves having Pryms of $3$-rank $1$ when $p=3$}

In Example \ref{Equadp3}, we showed that $\mathcal{R}_3^{(1,1)}$ and $\mathcal{R}_3^{(0,1)}$ are non-empty when $p=3$, by finding curves $X$ of $3$-rank $1$ or $0$ having an unramified double cover $\pi:Y \to X$ such that $P_\pi$ has $3$-rank $1$.
Here we give a second proof of this using the methods of this section.

For $t,u \in k$ with $t \neq u$, consider the genus 2 curve 
\[Z_{t,u}: z^2=D(x):=x^6+x^3+1+x(x^3+1)(tx+u)=x^6+tx^5+ux^4+x^3+tx^2+ux+1.\]
One can check that $Z_{t,u}$ is smooth if $t \not = u$ and 
$Z_{t,u}$ has $3$-rank $1$ for $t \neq \pm u$.


%

\begin{example} Let $p=3$.
Consider the plane quartic $X=V \cap K$ where $K$ is the Kummer surface of ${\rm Jac}(Z_{t,u})$ and 
$V \subset {\mathbb P}^3$ is a plane.
Then $X$ has an unramified double cover $\pi:Y \to X$ such that $P_\pi \simeq Z_{t,u}$ has
$p$-rank $1$.
\begin{enumerate}
\item If $V: -X_2+X_3+X_4=0$ 
and $t=1$, $u=0$, then $X$ is smooth with $3$-rank $f=1$.
\item If $V: -X_1-X_2+X_4=0$ and $t=0$, $u=1$, then $X$ is smooth with $3$-rank $f=0$. 
\end{enumerate}
\end{example}

\subsection{The moduli space of genus $3$ curves having Pryms of $3$-rank $0$ when $p=3$}
\label{Sp3fiber}

In this section, we fix $p=3$. 
In Section \ref{Sfamily3}, we parametrize ${\mathcal M}_2^0$ 
(the $3$-rank $0$ stratum of $\mathcal{M}_2$)
by a $1$-parameter family of curves of genus $2$ and $3$-rank $0$. 
Let $\alpha$ be the name of this parameter.
In Section \ref{S20}, we then analyze ${\rm det}(H_X)$ as $V$ and $\alpha$ vary.
This allows us to prove some information about the locus of the parameter space 
where $X$ is non-ordinary.
This implicitly provides geometric information about $\mathcal{R}_3^{(2,0)}$.

\subsubsection{A family of genus $2$ curves with $3$-rank $0$ when $p=3$} \label{Sfamily3}

For $\alpha \in k - \{0,1,-1\}$,   define
$$Z_\alpha: \ z^2=A(x)B(x),$$ where
\begin{align*}
A(x)= x^3-\alpha x^2+\alpha x+(\alpha+1), \hbox{ and } \ B(x)=(x-\alpha)(x-(\alpha+1))(\alpha x+(\alpha +1)).
\end{align*}

The importance of the next lemma is that it shows that ${\rm Jac}(Z_\alpha)$ parametrizes the $3$-rank~$0$ stratum of 
$\mathcal{A}_2$, which is irreducible of dimension $1$ when $p=3$.

\begin{lemma} When $p=3$, then
the family $\{Z_\alpha\}_\alpha$ is a non-isotrivial family of smooth curves of genus $2$ and $3$-rank $0$.
\end{lemma}

\begin{proof}
For $\alpha\in k - \{0,1,-1\}$, the polynomial $A(x)B(x)$ has no repeated roots, hence the curve $Z_\alpha$ is smooth and of genus $2$.
If $Z_\alpha$ is isomorphic to $Z_{\beta}$ for $\alpha, \beta\in k-\{0,1,-1\}$ then they have the same absolute Igusa invariants $j_1, j_2, j_3$ \cite{Igusa}. Using SageMath \cite{Sage}, we find that the absolute Igusa invariants of $Z_\alpha$ are: 
$$\begin{array}{c}
j_1=-\dfrac{(\alpha -1)^{6}}{\alpha(\alpha + 1)^2}, \\
j_2=-\dfrac{(\alpha -1)^6}{\alpha(\alpha + 1)^2  }, \\
j_3=\dfrac{(\alpha -1)^2 (\alpha^2 -\alpha -1)^2}{\alpha^2}.
\end{array}
$$
In particular, the absolute Igusa invariants are non-constant functions of $\alpha$, hence the family $\{Z_\alpha\}_\alpha$ is non-isotrivial.

By Proposition~\ref{prop:yui}, the Cartier-Manin matrix $M$ of $Z_{\alpha}$ is
\[\left( \begin{array}{ccc}
(\alpha+1)^3 & -(\alpha+1)^4  \\
1 & -(\alpha+1) 
\end{array}\right)\]
so the matrix $M^{(3)}M$ is 
\[\left( \begin{array}{ccc}
(\alpha+1)^9 & -(\alpha+1)^{12}  \\
1 & -(\alpha+1)^3 
\end{array} \right)\left( \begin{array}{ccc}
(\alpha+1)^3 & -(\alpha+1)^4  \\
1 & -(\alpha+1) 
\end{array}\right)= \left( \begin{array}{ccc}
0 & 0 \\
0 & 0 
\end{array} \right),\]
so $Z_\alpha$ has $3$-rank $f'=0$ by Proposition~\ref{prop:prank}. 
\end{proof}

\subsubsection{The locus of the parameter space where $X$ is non-ordinary when $p=3$}
\label{S20}

Now, we compute the equation of the Kummer surface $K_\alpha$ of $Z_\alpha$ and choose a plane $V=V_{a,b,c,d}$ to obtain 
the smooth plane quartic $X_V^\alpha=K_\alpha\cap V$.  
Our goal is to find information about the $3$-rank $f=f_V^\alpha$ of $X_V^\alpha$ as $V$ and $\alpha$ vary. 
To do this, we determine the Hasse-Witt matrix $H:=H_{X_V^\alpha}$ of $X_V^\alpha$ as in \eqref{HW3}. 

\begin{proposition} \label{Lplane3}
For a generic choice of plane $V$, the curve $X_V^\alpha$ is ordinary for all but finitely many $\alpha$ and 
has $3$-rank $2$ for at least one and at most finitely many $\alpha$.
\end{proposition}

\begin{proof}
When $d=1$, the value of the $3$-rank of $X_V^\alpha$ is determined by polynomial conditions in $a,b,c$ and $\alpha$.
In particular, $X_V^\alpha$ has $3$-rank $3$ if and only if the determinant of $H_V^\alpha$ is non-zero.  
So the first statement can be proven by checking, for a fixed plane $V$ and fixed parameter $\alpha$, whether ${\rm det}(H_V^\alpha) \not = 0$.
The second statement can be proven by checking, for a fixed plane $V$, whether 
${\rm det}(H_V^\alpha) = 0$ under a polynomial condition on $\alpha$ and whether there exists one value of $\alpha$
satisfying that polynomial condition for which $X_V^\alpha$ is smooth and has $3$-rank $2$.

Thus both statements follow from the next claim.

\smallskip

{\bf Claim:} For the plane $V: -X_2+X_4=0$, the curve $X_V^{\alpha}$ is ordinary for all but finitely many $\alpha \in k - \{0,1,-1\}$, and
$X_V^\alpha$ has $3$-rank $2$ for a non-zero finite number of $\alpha \in k$.

\smallskip

{\bf Proof of claim:}
When $V: -X_2+X_4=0$, the Hasse-Witt matrix $H$ of $X_V^\alpha$ is
\[\left( \begin{array}{ccc}
a_{11} & a_{12} & a_{13}\\
a_{21} & a_{22} & a_{23} \\
a_{31} & a_{32} & a_{33}  
\end{array}\right)\]
where 
  \begin{eqnarray*}
 a_{11} &=& \alpha^{13} - \alpha^{11} - \alpha^{10} + \alpha^9 + \alpha^7 + \alpha^6 - \alpha^3 - \alpha^2 - 1,\\
 a_{12}&=& -\alpha^7 - \alpha^6 + \alpha^5 + \alpha^4 + \alpha^2 + \alpha,\\
 a_{13}&=& \alpha^{10} + \alpha^9 + \alpha^7 - \alpha^6 + \alpha^5 - \alpha^4, \\
 a_{21} &=& -\alpha^{16} - \alpha^{13} + \alpha^{11} + \alpha^9 + \alpha^8 + \alpha^7 - \alpha^5 + \alpha^4 - \alpha^3 - \alpha^2,\\
 a_{22} &=& \alpha^{13} + \alpha^9 + \alpha^8 + \alpha^7 - \alpha^6 - \alpha^4 - \alpha^3 - \alpha^2 - \alpha - 1, \\
 a_{23} &=& -\alpha^{13} + \alpha^{10} + \alpha^9 - \alpha^6 + \alpha^5 + \alpha^4 + \alpha^3 - \alpha^2 - \alpha - 1, \\ 
 a_{31} &=& \alpha^{13} + \alpha^{12} - \alpha^9 + \alpha^8 + \alpha^7 + \alpha^6 - \alpha^5 + \alpha^2 + \alpha, \\
 a_{32} &=& \alpha^{10} - \alpha^8 + \alpha^7 + \alpha^5 - \alpha^4 - \alpha^3 - \alpha^2 - \alpha - 1, \\
 a_{33} &=& \alpha^{12} - \alpha^{10} + \alpha^6 + \alpha^5 - \alpha^4 - \alpha - 1.
 \end{eqnarray*}

The determinant $\text{Det}_{H}$ of $H$ is
\begin{eqnarray*}
\text{Det}_H=\alpha^3 (\alpha + 1)^4 (\alpha -1)^5 (\alpha^3 + \alpha^2 + \alpha -1)  (\alpha^5 + \alpha^4 + \alpha^3 + \alpha^2 -\alpha + 1)\\ (\alpha^5 -\alpha^4 + \alpha^3 + \alpha^2 + \alpha + 1) (\alpha^6 -\alpha^4 + \alpha^3 -\alpha + 1) (\alpha^7 + \alpha^5 + \alpha -1).
\end{eqnarray*}

For each $\alpha\in k-\{0,1,-1\}$ which is not a root of $\text{Det}_{H}$, the Hasse-Witt matrix is invertible and so $X_V^\alpha$ is ordinary. For $\alpha\in k$ satisfying $\alpha^3 + \alpha^2 + \alpha -1=0$, a computation in SageMath \cite{Sage} shows that $X_V^\alpha$ is smooth and the matrix $HH^{(3)}H^{(3^2)}$ has rank $2$. Therefore, for these values of $\alpha$, $X_V^\alpha$ has $3$-rank $2$ by Proposition~\ref{prop:prank}.


\end{proof}

Proposition \ref{Lplane3} does not eliminate the possibility that there exists a plane $V$ 
such that ${\rm det}(H_V^\alpha)=0$ for all $\alpha$.

\begin{proposition} \label{Pfixalpha}
For a generic choice of $\alpha \in k$, the curve $X_V^\alpha$ is ordinary for a generic choice of $V$ and 
has $3$-rank $2$ under a codimension $1$ condition on $V$.
\end{proposition}

\begin{proof}
The first statement already follows from Proposition \ref{Lplane3}.
The second statement can be proven by checking, for fixed $\alpha \in k$, whether 
${\rm det}(H_V^\alpha) = 0$ under a polynomial condition on $a,b,c$ and whether there exists one possibility for $a,b,c$
satisfying that polynomial condition for which $X_V^\alpha$ is smooth and has $3$-rank $2$.

Let $\alpha \in {\mathbb F}_9$ be fixed to be a root of the polynomial $t^2 + 2t + 2$.  If $d=1$, the Hasse-Witt matrix $H=(a_{ij})_{i,j}$ of $X_V^{\alpha}$, for arbitrary $a$, $b$, $c$, is given by~:
$$\begin{array}{l}
a_{11}= a^3c +b^2+ac+ (\alpha + 1)(a^3-bc+a) + (\alpha - 1)(ab-b)-\alpha c^2      \\
a_{12}=b^3c + (\alpha + 1)b^3      \\    
a_{13}=c^4 - ac + (\alpha + 1)c^2(c-1) +  (-\alpha + 1)b^2 + -\alpha (ab+ bc) \\
a_{21}=a^3b -ab + (-\alpha - 1)(a^3+ac+c) + (-\alpha + 1)(a^2+c^2+a-bc)    \\
a_{22}= b^4 + (-\alpha - 1)b^3 -\alpha c^2 + \alpha b      \\
a_{23}= bc^3 + (-\alpha - 1)c^3 + \alpha (a^2 -ac+bc+c^2)+ (\alpha - 1)ab  \\
a_{31}=a^4 - a^2+ (-\alpha + 1)+ (\alpha + 1)(a^3-ab-b )+ (-\alpha + 1)(b^2-bc-c) + \alpha ac     \\
a_{32}=ab^3 + (\alpha + 1)b^3 + \alpha bc  \\
a_{33}=ac^3 + a^2+ (\alpha + 1)(c^3+ac -c) - \alpha (b^2+bc+ab).
\end{array}$$
Then one can check that ${\rm det}(H)$ is non-vanishing in $a,b,c$. 
Also, when $(a,b,c)=(2,0,2)$, then one can check that 
the curve $X_V^\alpha$ is smooth with $3$-rank $2$.
\end{proof}

We note that Proposition \ref{Pfixalpha} does not eliminate the possibility that there exists $\alpha \in k$ 
such that ${\rm det}(H_V^\alpha)=0$ for all planes $V$.

\section{Points on the Kummer surface} \label{Skummer}

Suppose that $Z$ is a supersingular curve of genus $2$ defined over a finite field ${\mathbb F}_q$ 
of characteristic $p$.
This section contains a result about the number of ${\mathbb F}_q$-points on the Kummer surface $K$ of ${\rm Jac}(Z)$.
The material in this section is probably well known to experts but we could not find it in the literature.
The connection between this section and the rest of the paper is found in Question \ref{Qss}:
if $X=V \cap K$ for some plane $V$ and if $p$ divides $\#X({\mathbb F}_q)$, 
then the $p$-rank of $X$ is at least $1$.

Let $Z$ be a genus $2$ curve over ${\mathbb{F}}_q$.  Suppose that $Z$ has equation $z^2=D(x)$ and define a quadratic twist $W$ of $Z$ by $\lambda z^2=D(x)$ for $\lambda\in\mathbb{F}_{q}^\times\setminus (\mathbb{F}_{q}^\times)^2$. The isomorphism class of $W$ does not depend on the choice of non-square element $\lambda$.

\begin{lemma}\label{lem:Kummer points}
Let $Z$ be a genus $2$ curve over ${\mathbb{F}}_q$ and let $W$ be its quadratic twist.
Let $K=\Jac(Z)/\langle -1 \rangle$.
Then \[|K(\mathbb{F}_q)|=(|{\rm Jac}(Z)(\mathbb{F}_q)|+|{\rm Jac}(W)(\mathbb{F}_q)|)/2.\]
\end{lemma}

\begin{proof}
The degree $2$ cover $\phi: \Jac(Z)\to K$ is defined over $\bbF_q$.
Let $\psi:\Jac(Z)\to \Jac(W)$ be the isomorphism of abelian varieties over $\bbF_{q^2}$ induced by the isomorphism of the underlying curves given by $(x,z)\mapsto (x,\sqrt{\lambda}z)$. 
Let $\tau$ be a generator for $\Gal(\bbF_{q^2}/\bbF_q)$. Then $\tau\psi\tau^{-1}=-\psi$. 
 
Let $P\in K(\bbF_q)$.
Write $\phi^{-1}(P)=\{Q,-Q\}$. Since $P\in K(\bbF_q)$ and $\phi$ is defined over $\bbF_q$, then 
$\{\sigma(Q),-\sigma(Q)\}=\{Q,-Q\}$ 
for all $\sigma\in\Gal(\overline{\bbF}_q/\bbF_q)$. Therefore, $Q$ is defined over $\bbF_{q^2}$ and either $\tau(Q)=Q$, whereby $Q\in \Jac(Z)(\bbF_q)$, or $\tau(Q)=-Q$, whereby $\psi(Q)\in \Jac(W)(\bbF_q)$. The points $P\in K(\bbF_q)$ for which $\phi^{-1}(P) =\{Q,-Q\}$ with $Q=-Q$ are precisely those for which $Q\in \Jac(Z)(\bbF_q)$ and $\psi(Q)\in \Jac(W)(\bbF_q)$. Therefore, every point in $K(\bbF_q)$ is counted twice in $|\Jac(Z)(\bbF_q)|+|\Jac(W)(\bbF_q)|$.
\end{proof}

The zeta function of a genus $2$ curve $Z/\bbF_q$ is 
\[\mathcal{Z}(T)=\exp\left(\sum_{k\geq 1}{\frac{|Z(\bbF_{q^k})|}{k}T^k} \right)=\frac{L_{Z/\bbF_q}(T)}{(1-T)(1-qT)},\]
where $L_{Z/\bbF_q}(T)=1+a_1T+a_2T^2+qa_1T^3+q^2T^4=\prod_{i=1}^4{(1-\alpha_iT)}$ with $\alpha_1\alpha_3=\alpha_2\alpha_4=q$.

\begin{lemma}
Let $Z$ be a genus $2$ curve over ${\mathbb{F}}_q$, let $A={\rm Jac}(Z)$ and let $K=A/\langle -1 \rangle$. Then 
\[|{\rm Jac}(Z)(\mathbb{F}_q)|=1+a_1+a_2+a_1q+q^2\]
and
\[|K(\mathbb{F}_q)|=1+a_2+q^2\]
where the $a_i$ are the coefficients of $L_{Z/\bbF_q}(T)$ as defined above.
\end{lemma}

\begin{proof}
The second statement follows from the first, using Lemma \ref{lem:Kummer points} and the fact that if $W/\bbF_q$ is the quadratic twist of $Z$, then $L_{W/\bbF_q}(T)=L_{Z/\bbF_q}(-T)=1-a_1T+a_2T^2-qa_1T^3+q^2T^4$. 

For the first statement, note that
\begin{equation}\label{eq:Jac points}
|{\rm Jac}(Z)(\mathbb{F}_q)|=\frac{|Z(\bbF_q)|^2+|Z(\bbF_{q^2})|}{2}-q.
\end{equation}
Equating the coefficients of $T$ and $T^2$ in
\[L_{Z/\bbF_q}(T)=(1-T)(1-qT)=\exp\left(\sum_{k\geq 1}{\frac{|Z(\bbF_{q^k})|}{k}T^k} \right)\]
gives $a_1=|Z(\bbF_q)|-(q+1)$ and $a_2=\frac{1}{2}|Z(\bbF_q)|^2+\frac{1}{2}|Z(\bbF_{q^2})|-(q+1)|Z(\bbF_q)|+q$. The result now follows from \eqref{eq:Jac points}.
\end{proof}

\begin{corollary}
Let $Z$ be a supersingular genus $2$ curve over ${\mathbb{F}}_q$, let $A={\rm Jac}(Z)$ and let $K=A/\langle -1 \rangle$. Then 
\[|K(\mathbb{F}_q)|\equiv 1\bmod{q}.\]
\end{corollary}

\begin{proof}
If $Z$ is supersingular then $q\mid a_2$. The result now follows.
\end{proof}

\begin{question} \label{Qss}
Suppose $Z$ is supersingular.  
Does there exist a plane $V \subset {\mathbb P}^3$ defined over ${\mathbb F}_q$ such that 
$p$ divides $\#X({\mathbb F}_q)$ where $X=V \cap K$?
If so, then the $p$-rank of $X$ is at least $1$.
\end{question}

\section{Results for arbitrary $g$}\label{sec:induction}

In this section, when $3 \leq p \leq 19$, we use the results from Section \ref{sec:forwards}
about genus $3$ curves in characteristic $p$
to verify the existence of smooth curves $X$ of arbitrary genus $g \geq 3$ having an unramified double cover 
whose Prym has small $p$-rank.  Specifically, we work inductively to 
study the dimension of certain moduli strata $\mathcal{R}_g^{(f,f')}$ 
for $g\geq 3$ in characteristic $p$ with $3\leq p\leq 19$.
The reader is strongly advised to read \cite{OzmanPries2016} before reading this section.

A highlight of this approach is that $X$ is smooth and we can control its $p$-rank $f$.
Indeed, the original proof of \cite{Wirtinger}, found in \cite[Section 2]{MR1013156}, 
shows that $\bar{\mathcal{R}}_g^{(f'+1, f')}$ is non-empty for each $g \geq 2$ and $0 \leq f' \leq g-1$;
in other words, there is a \emph{singular} curve of genus $g$ and $p$-rank $f'+1$ with an unramified double cover 
$\pi$ such that $P_\pi$ has $p$-rank $f'$.
We omit the details of this argument.

In this section, the word \emph{component} means irreducible component.
Although the phrasing is slightly redundant, we emphasize that a component of a given dimension is 
\emph{non-empty} because this property of the component is the most difficult to prove 
and is sufficient to yield the existence applications.

\subsection{Increasing the $p$-rank of the Prym variety}

The next result allows us to use geometric information about $\mathcal{R}_g^{(f,f')}$
to deduce geometric information about $\mathcal{R}_g^{(f,F')}$ when $f' \leq F' \leq g-1$. 

\begin{proposition} \label{Pgoup} \cite[Proposition 5.2]{OzmanPries2016}
Let $g \geq 3$.
Suppose that $\mathcal{R}_g^{(f, f')}$ is non-empty and has a component of dimension $g-2 +f +f'$
in characteristic $p$.  
Then $\mathcal{R}_g^{(f, F')}$ is non-empty and has a component of dimension $g-2 +f+F'$
in characteristic $p$
for each $F'$ such that $f' \leq F' \leq g-1$.
\end{proposition}

\subsection{Background on boundary of $\mathcal{R}_g$} \label{Sboundary}

The strategy used in \cite{OzmanPries2016} is to use unramified double
covers of \emph{singular} curves of given genus and $p$-rank to produce 
unramified double covers of \emph{smooth} curves of the same genus and $p$-rank, 
with control over the $p$-rank of the Prym variety.
This strategy must be implemented very precisely because, in general, the $p$-rank of both the curve and 
the Prym will increase when deforming $X$ away from the boundary.
In fact, there are situations where this is guaranteed to happen.

This section contains background about $p$-ranks of unramified double covers of singular curves.
Let $\bar{\mathcal{R}}_g$ be the compactification of $\mathcal{R}_g$ as defined and analyzed in \cite[Section 1.4]{F_generalL}.
The points of  $\bar{\mathcal{R}}_g \backslash \mathcal{R}_g$ represent unramified double covers of 
singular stable curves of genus $g$.

Let $\bar{\mathcal{R}}_{g;1}=\bar{\mathcal{R}}_g \times_{\bar{\mathcal{M}}_g} \bar{\mathcal{M}}_{g;1}$ be the moduli space whose points represent unramified double covers 
$\pi:Y  \to X$ together with marked points $y \in Y$ and $x \in X$
such that $\pi(y)=x$, as in \cite[Section 2.3]{OzmanPries2016}.
Adding a marking increases the dimension of the moduli space by $1$.
By \cite[Lemma 2.1]{OzmanPries2016}, 
there is a surjective morphism $\psi_R: \bar{\mathcal{R}}_{g;1} \to \bar{\mathcal{R}}_g$ whose fibers are irreducible.

Suppose that $g=g_1+g_2$, with $g_i \geq 1$.  We recall some material
about the boundary divisor $\Delta_{g_1:g_2}[\bar {\mathcal R}_{g}]$ from \cite[Section 1.4]{F_generalL}. 
This boundary divisor
is the image of the clutching map
\[\kappa_{g_1:g_2}: \bar {\mathcal R}_{g_1;1} \times \bar {\mathcal R}_{g_2; 1} \rightarrow \bar {\mathcal R}_{g},\] 
defined on a generic point as follows.  
Let $\tau_1$ be a point of $\bar {\mathcal R}_{g_1;1}$ 
representing $(\pi_1:C'_1 \to C_1, x'_1 \mapsto x_1)$
and let $\tau_2$ be a point of $\bar {\mathcal R}_{g_2;1}$ representing 
$(\pi_2:C'_2 \to C_2, x'_2 \mapsto x_2)$.
Let $X$ be the curve with components $C_1$ and $C_2$, formed by identifying $x_1$ and $x_2$ in an ordinary double point.
Let $Y$ be the curve with components $C'_1$ and $C'_2$, formed by identifying
$x_1'$ and $x_2'$ (resp.\ $x_1''=\sigma(x'_1)$ and $x_2''=\sigma(x'_2)$) in an ordinary double point. 
Then $\kappa_{g_1:g_2}(\tau_1, \tau_2)$ is the point representing the unramified double cover $Y \to X$.
This is illustrated in Figure \ref{fig:Deltagi}. 

\begin{figure}[h]
\centering
\includegraphics[scale=0.5]{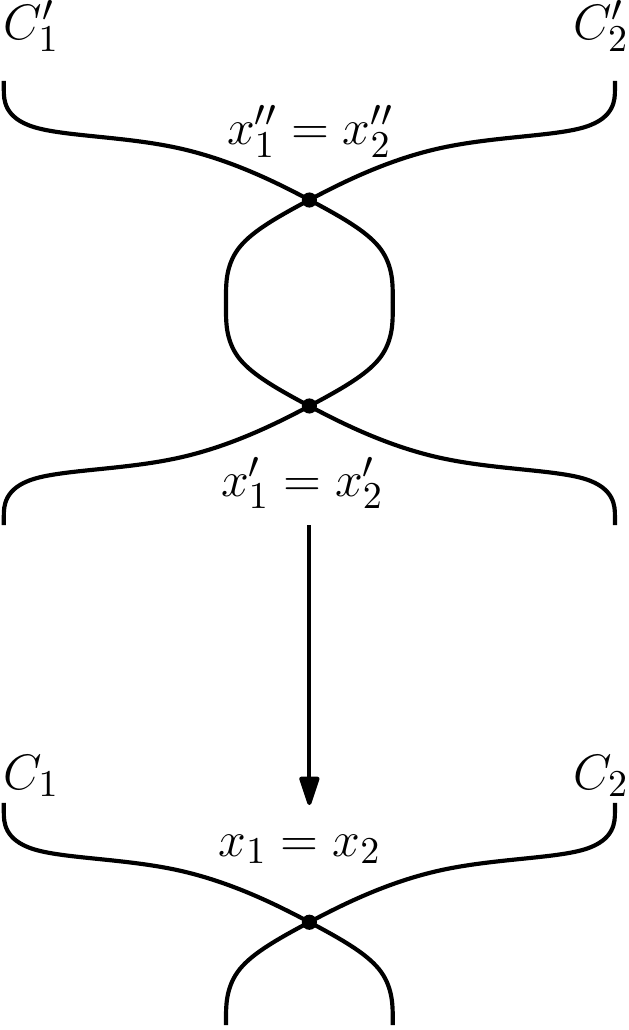}
\caption{$\Delta_{g_1:g_2}$}
\label{fig:Deltagi}
\end{figure}

In \cite[Section 3.4.1]{OzmanPries2016}, the authors analyze the $p$-rank stratification of this boundary divisor.
By \cite[Lemmas 3.6-3.7]{OzmanPries2016}, 
the clutching morphism restricts to the following:
\begin{equation}\label{Eclutch}
\kappa_{g_1:g_2}: \mathcal{R}_{g_1;1}^{(f_1,f'_1)} \times \mathcal{R}_{g_2;1}^{(f_2,f'_2)} \rightarrow \Delta_{g_1:g_2}[\bar{\mathcal{R}}_{g}^{(f_1+f_2,f'_1+f'_2+1)}].
\end{equation}

The following lemma is useful in the inductive arguments
in Section \ref{Sfinalresult}.

\begin{lemma} \label{Linductive}
Suppose that $S_i \subset \mathcal{R}_{g_i}^{(f_i,f'_i)}$ has dimension $d_i$ for $i=1,2$.
Then the dimension of
\[\mathcal{K}=\kappa_{g_1:g_2}(\psi_{R}^{-1}(S_1) \times \psi_{R}^{-1}(S_2))\]
is $d_1+d_2 + 2$.
Furthermore, if $S_i$ is a component of $\mathcal{R}_{g_i}^{(f_i,f'_i)}$ for $i=1,2$, then 
$\mathcal{K}$ is contained in a component of 
$\bar{\mathcal{R}}_{g}^{(f_1+f_2,f'_1+f'_2+1)}$ whose dimension is at most $d_1+d_2 + 3$.
\end{lemma}

\begin{proof}
The first statement follows from the facts that the fibers of $\psi_R$ have dimension $1$ 
and the clutching maps are finite.
The second statement follows from \eqref{Eclutch} and \cite[page 614]{V:stack}, 
see \cite[Lemma 3.1]{OzmanPries2016}.
\end{proof}

\subsection{Some extra results when $p=3$} \label{Sextrap=3}

The $p=3$ case is guaranteed to be more difficult, because $\mathcal{R}_{2}^{(0,0)}$ and $\mathcal{R}_2^{(1,0)}$ are empty in that case \cite[Section 7.1]{FaberVdG2004}.
In other words, when $p=3$, if $\pi:Y \to X$ is an unramified double cover of a genus $2$ curve such 
that $P_\pi$ is non-ordinary, then $X$ is ordinary.
This is the key reason why there are extra hypotheses when $p=3$ in 
\cite[Propositions 6.1, 6.4, Theorem 7.2]{OzmanPries2016}.

We now have the extra information that all pairs $(f,f')$ occur when $p=3$ and $g=3$.
In this section, we use this to confirm that the extra hypotheses when $p=3$ can be removed from most of the
results of \cite[Sections 6 and 7]{OzmanPries2016}.
This will allow us to work more uniformly for odd $p$ in the next section.

Let $\tilde{\mathcal{A}}_{g-1}$ denote the toroidal compactification of $\mathcal{A}_{g-1}$.
Let $\tilde{\mathcal{A}}_{g-1}^{f'}$ denote the $p$-rank $f'$ stratum of $\tilde{\mathcal{A}}_{g-1}$.
Let $V_g^{f'} = \bar{Pr}_g^{-1}(\tilde{\mathcal{A}}_{g-1}^{f'}) \cap \mathcal{R}_g$, 

The next result about the $p$-rank stratification of $R_3^{(f,1)}$ is true for all $p \geq 5$ by 
\cite[Proposition 6.4]{OzmanPries2016}.

\begin{lemma} \label{L3ok1}
The result \cite[Proposition 6.4]{OzmanPries2016} does not require the hypothesis $f \not = 0, 1$ when $p = 3$.
In other words, if $p=3$ and $0 \leq f \leq 3$, then $\Pi^{-1}(\mathcal{M}_3^f)$ is irreducible and $\mathcal{R}_3^{(f,1)} = \Pi^{-1}(\mathcal{M}_3^f) \cap V_3^1 $ is non-empty with dimension $2+f$.
\end{lemma} 

\begin{proof}
The only place that this hypothesis was used was to verify that $\mathcal{R}_3^{(f,1)}$ is non-empty, which we have now verified in Proposition \ref{R3f0_forwardmethod}.
\end{proof}

The next result about non-ordinary Prym varieties is true for all $p \geq 5$ by \cite[Theorem 7.1]{OzmanPries2016}.
We say that $P_\pi$ is almost ordinary if its $p$-rank satisfies $f'=g-2$. 

\begin{lemma}
The result \cite[Theorem 7.1]{OzmanPries2016} does not require the hypothesis $f \geq 2$ when $p=3$ and $g \geq 3$.
In other words, if $p=3$, $g \geq 3$, and $0 \leq f \leq g$, 
then $\mathcal{R}_g^{(f, g-2)}$ is non-empty and each of its components has dimension $2g-4+f$. More generally, let $S$ be a component of $\mathcal{M}_g^f$, then the locus of points of $\Pi^{-1}(S)$ representing unramified double covers for which the Prym variety $P_{\pi}$ is almost ordinary is non-empty and codimension 1 in $\Pi^{-1}(S)$.
\end{lemma}

\begin{proof}
The only place this hypothesis was used was to apply \cite[Proposition 6.4]{OzmanPries2016} in the `base cases' and `non-empty' paragraphs.  So the result follows by Lemma \ref{L3ok1}.
\end{proof}

The hypothesis $p \geq 5$ also appears in \cite[Corollary 7.3]{OzmanPries2016}, because a key point of the proof is that $\mathcal{R}_2^{(1,0)}$ and 
$\mathcal{R}_2^{(0,0)}$ are non-empty, which is false when $p =3$.
We generalize \cite[Corollary 7.3]{OzmanPries2016} to include the case $p=3$ in Section \ref{Sfinalresult}.

\subsection{A dimension result}

The following result is also needed in Section \ref{Sfinalresult}.

\begin{proposition} \label{P20in30}
Let $3 \leq p \leq 19$.  Then $\mathcal{R}_3^{(3,0)}$ contains a component of 
dimension $4$ and $\mathcal{R}_3^{(2,0)}$ contains a component of dimension $3$.
\end{proposition}

\begin{proof} 
Let $3 \leq p \leq 19$ and $f=2,3$.
By Proposition~\ref{R3f0_forwardmethod}, there exists an unramified double cover
$\pi_f: Y_f \to X_f$ such that $X_f$ is a smooth plane quartic with $p$-rank f and $P_{\pi_f}$ has $p$-rank $0$.
This shows that $\mathcal{R}_3^{(f,0)}$ is non-empty.
Let $S_f$ denote a component of $\mathcal{R}_3^{(f,0)}$ containing the point representing $\pi_f$.
By purity, ${\rm dim}(S_f) \geq 1+f$.  
To finish the proof, it suffices to show that ${\rm dim}(S_f) = 1+f$. 

Consider the morphism ${Pr}_3: {\mathcal{R}}_3 \to {\mathcal{A}}_2$.
Let $\mathcal{Z}$ be an irreducible component of the $p$-rank $0$ stratum ${\mathcal{A}}_2^0$ of ${\mathcal{A}}_2$.
By \cite[Theorem 7, page 163]{MR0414557}, $\mathcal{Z}$ has dimension $1$ (in fact, it is rational \cite[page 117]{Oort}).
As in Section \ref{SVerra}, ${Pr}_3^{-1}(\mathcal{Z})$ has one irreducible component $N_{\mathcal Z}$ of relative dimension $3$ 
and three components of relative dimension $2$.  
Since ${\rm dim}(\mathcal{Z})=1$, it follows that ${\rm dim}(N_{\mathcal Z})=4$.

Let $f=3$.  Since $X_3$ is a smooth plane quartic, then $\pi_3$ is represented by a point 
of $N_{\mathcal Z}$ for some component $\mathcal{Z}$ of ${\mathcal{A}}_2^0$.
Thus $S_3 \subset N_{\mathcal Z}$. 
This implies that ${\rm dim}(S_3) \leq 4$, which completes the proof when $f=3$.

In fact, for any component $\mathcal{Z}$ of ${\mathcal{A}}_2^0$ containing the point representing $P_{\pi_3}$, 
the fact that the smooth plane quartic $X_3$ has $p$-rank $3$ 
implies that the generic point of $N_{\mathcal Z}$ is in $\mathcal{R}_3^{(3,0)}$.
This is relevant when $P_{\pi_3}$ is superspecial, 
in which case the point that represents it is in multiple components $\mathcal{Z}$.

Let $f=2$.  Since $X_2$ is a smooth plane quartic, then $\pi_2$ is represented by a point 
of $N_{\mathcal{Z}'}$ for some component $\mathcal{Z}'$ of ${\mathcal{A}}_2^0$.  Thus $S_2 \subset N_{\mathcal{Z}'}$.
We claim that the generic point of $N_{\mathcal{Z}'}$ represents an unramified double cover where $X$ has $p$-rank $3$, not $2$.
This is clear when $3 \leq p \leq 11$, because 
${\mathcal{A}}_2^0$ is irreducible for those primes 
by \cite[Theorem 5.8]{katsuraoort} so there is only one component of $\mathcal{A}_2^0$, so $\mathcal{Z}'=\mathcal{Z}$.
For $13 \leq p \leq 19$, we note in the tables given in Section \ref{tables} 
that the same curve $Z$ is used in the cases $(3,0)$ and $(2,0)$.
This means that the Prym varieties $P_{\pi_2}$ and $P_{\pi_3}$ are isomorphic.  
So ${\mathcal{Z}'}$ is one of the components $\mathcal{Z}$ of $\mathcal{A}_2^0$ containing the point representing $P_{\pi_3}$,
and the claim is true by the previous paragraph. 
Since $N_{\mathcal{Z}'}$ has dimension $4$ and its generic point represents a cover where $X$ has $p$-rank $3$, 
it follows that ${\rm dim}(S_2) \leq 3$.
\end{proof}

\subsection{Final result} \label{Sfinalresult}

In this section, in characteristic $3 \leq p \leq 19$,
we verify the existence of unramified double covers $\pi:Y \to X$ 
where $X$ has genus $g$ and $p$-rank $f$ and $P_\pi$ has 
$p$-rank $f'$, for arbitrary $g$ as long as $f$ is bigger than approximately $2g/3$ and 
$f'$ is bigger than approximately $g/3$.
This is most interesting when either $\frac{g}{3} \leq f' < \frac{g}{2}-1$ or $p=3$
because \cite[Corollary 7.2]{OzmanPries2016} resolves the case when 
$\frac{g}{2}-1 \leq f' \leq g-1$ with no conditions on $g$ and $f$ when $p \geq 5$.

We first include an inductive result which holds for any odd prime $p$. 
This strengthens \cite[Theorem 7.2]{OzmanPries2016}.

\begin{theorem} \label{Tbigg}
Let $f_0$ be such that $\mathcal{R}_3^{(f_1, 0)}$ has a (non-empty) component of dimension $1 + f_1$
in characteristic $p$ for each $f_1$ such that $f_0 \leq f_1 \leq 3$. 

Let $g \geq 2$ and write $g=3r+2s$ for integers $r,s \geq 0$.
Suppose that $rf_0 \leq f \leq g$ (with $f \geq rf_0 + 2s$ when $p=3$)
and $r+s-1 \leq f' \leq g-1$.

Then $\mathcal{R}_g^{(f, f')}$ has a (non-empty) component of dimension $g-2 +f +f'$
in characteristic $p$.
\end{theorem}

\begin{proof}
In light of Proposition \ref{Pgoup}, it suffices to prove the result when $f' =r+s-1$.
The proof is by induction on $r+s$.
In the base case $(r,s)=(0,1)$, then $g=2$ and the result is true by \cite[Proposition 6.1]{OzmanPries2016}.
In the base case $(r,s)=(1,0)$, then $g=3$ and the result is true by hypothesis. 
Now suppose that $r + s \geq 2$.
As an inductive hypothesis, suppose that the result is true for all pairs $(r',s')$ such that $1 \leq r'+s' < r+s$.

Case 1: suppose that $r \geq 1$.  This implies $g \geq 5$.  Let $g_1=3$ and $g_2=g-3$.
By the hypotheses on $f$, it is possible to choose $f_1, f_2$ such that $f_1+f_2=f$ and $f_0 \leq f_1 \leq 3$ 
and $f_0 (r-1) \leq f_2 \leq g_2$ (with $f_2 \geq f_0(r-1) + 2s$ if $p=3$).
Let $f_1'=0$ and $f_2'=r+s-2$.
By hypothesis, $\mathcal{R}_3^{(f_1,0)}$ is non-empty and has a component $S_1$ 
of dimension $d_1= 1+f_1$.
By the inductive hypothesis applied to $(r-1, s)$, it follows that $\mathcal{R}_{g_2}^{(f_2, r+s-2)}$ is non-empty and 
has a component $S_2$ of dimension $d_2= g_2-2 +f_2 + (r+s -2)$.

By Lemma \ref{Linductive}, the image $\mathcal{K}$ of the clutching morphism
$\kappa$ has dimension $d_1+d_2+2$.
Furthermore, $\mathcal{K}$ is contained in a component $\mathcal{W}$ of 
$\bar{\mathcal{R}}_{g}^{(f_1+f_2,f'_1+f'_2+1)}= \bar{\mathcal{R}}_{g}^{(f,r+s-1)}$ 
whose dimension is at most 
\[d_1+d_2 + 3= (g_2+3) -2 + (f_1+f_2) + (r+s-1) = g-2+f +f'.\]
Also ${\rm dim}(\mathcal{W}) \geq g-2+f +f'$ by purity.
Thus ${\rm dim}(\mathcal{W}) = g-2+f +f'$ and the generic point of $\mathcal{W}$ is not contained in $\mathcal{K}$.
The generic points of $S_1$ and $S_2$ represent unramified double covers of smooth curves by hypothesis.
Thus the generic point of $\mathcal{W}$ is not contained in any other boundary component of 
$\bar{\mathcal{R}}_g$.  
It is thus contained in $\mathcal{R}_g$ and 
and so it represents an unramified double cover of a smooth curve.

Case 2: suppose that $s \geq 1$.  This implies $g\geq 4$.  Let $g_1=2$ and $g_2=g-2$.
By the hypotheses on $f$, when $p \geq 5$, it is possible to choose $f_1, f_2$ such that $f_1+f_2=f$ and 
$0 \leq f_1 \leq 2$ and $f_0 r \leq f_2 \leq g_2$.
When $p=3$, let $f_1=2$ and $f_2=f_0r + 2(s-1)$.
Let $f_1'=0$ and $f_2'=r+s-2$.
By \cite[Proposition 6.1]{OzmanPries2016}, $\mathcal{R}_2^{(f_1,0)}$ is non-empty and has a component $S_1$ 
of dimension $d_1= f_1$.
By the inductive hypothesis applied to $(r, s-1)$, it follows that $\mathcal{R}_{g_2}^{(f_2, r+s-2)}$ is non-empty and 
has a component $S_2$ of dimension $d_2= g_2-2 +f_2 + (r+s -2)$.
The rest of the proof follows the same reasoning as in Case (1).

\end{proof}

\begin{corollary} \label{Cmaincor}
Let $f_0=2$ and $3 \leq p \leq 19$.
Let $g \geq 2$ and write $g=3r+2s$ for integers $r,s \geq 0$.
Suppose that $2r \leq f \leq g$ (with $f \geq 2r + 2s$ when $p=3$)
and $r+s-1 \leq f' \leq g-1$.

Then $\mathcal{R}_g^{(f, f')}$ has a (non-empty) component of dimension $g-2 +f +f'$
in characteristic $p$.

In particular, this holds in the following situations:
\begin{enumerate}
\item If $g=3r$ and $(f, f')$ is such that $2 r \leq f \leq g$ and $r -1 \leq f' \leq g-1$;
\item If $g=3r + 2$ and $(f,f')$ is such that 
$2r \leq f \leq g$ and $r \leq f' \leq g-1$, 
(with $f \geq 2r+2$ when $p=3$);
\item If $g=3r + 4$ and $(f, f')$ is such that 
$2 r \leq f \leq g$ 
(with $f \geq 2r+4$ when $p=3$)
and $r+1 \leq f' \leq g-1$.
\end{enumerate}
\end{corollary}

\begin{proof}
By Proposition \ref{P20in30}, the hypothesis in Theorem \ref{Tbigg} holds for $f_0=2$.
The result is then immediate from Theorem \ref{Tbigg}.
\end{proof}

\bibliographystyle{plain}
\bibliography{prym}

\end{document}